\documentclass[]{article}

\usepackage{enumitem}
\usepackage{lipsum}
\usepackage{graphicx}
\usepackage{epstopdf}
\usepackage{algorithmic}
\ifpdf
  \DeclareGraphicsExtensions{.eps,.pdf,.png,.jpg}
\else
  \DeclareGraphicsExtensions{.eps}
\fi
\usepackage{amssymb}

\usepackage{amsthm}
\usepackage{amssymb}
\usepackage{bm}
\usepackage{graphicx,epstopdf}
\usepackage{nicefrac}
\usepackage{a4wide}
\usepackage{psfrag}
\usepackage{doi}

\usepackage[all,cmtip,line]{xy}
\usepackage[normalem]{ulem}
\usepackage{tikz,tikz-cd}
\usetikzlibrary{cd}
\usetikzlibrary{matrix, arrows}
\usepackage{pgfplots}
\usepackage{subcaption}
\usepackage{arcs}
\usepackage[english]{babel}
\usepackage[utf8]{inputenc}
\usepackage{lmodern}
\usepackage{hyperref,url}

\usepackage{mathtools}
\usepackage{cleveref}

\newtheorem{theorem}{Theorem}[section]
\newtheorem{lemma}[theorem]{Lemma}

\newtheorem{proposition}[theorem]{Proposition}
\newtheorem{problem}[theorem]{Problem}
\newtheorem{definition}[theorem]{Definition}
\newtheorem{assumption}[theorem]{Assumption}
\newtheorem{example}[theorem]{Example}
\newtheorem{remark}[theorem]{Remark}

\usepackage[T3,T1]{fontenc}
\newcommand{\modulo}[1]{\left\vert #1\right\vert}
\newcommand{\co}{{\mathsf{co}}}
\newcommand{\pl}{{\mathsf{p}}}

\renewcommand{\det}[1]{{\textnormal{det}(#1)}}
\newcommand{\balpha}{{\bm{\alpha}}}

\newcommand{\seminorm}[2]{\left\vert #1 \right\vert_{#2}}

\newcommand{\curl}{\operatorname{\mathbf{curl}}}
\renewcommand{\div}{\operatorname{div}}
\newcommand{\p}[1]{\langle #1\rangle}
\newcommand{\pr}[1]{\left( #1\right)}

\newcommand{\hcurl}[1]{\bm{H}( \curl; #1 )}
\newcommand{\hcurlcurl}[2]{\bm{H}_{#2}( \curl\curl; #1 )}
\newcommand{\hscurl}[2]{\bm{H}^{#2}( \curl; #1 )}
\newcommand{\hocurl}[1]{\bm{H}_0(\curl; #1 )}

\newcommand{\Hsob}[2]{\bm{H}^{#1}( #2 )}
\newcommand{\hsob}[2]{{H}^{#1}( #2 )}

\newcommand{\Lp}[2]{\bm{L}^{#1}( #2 )}
\newcommand{\lp}[2]{{L}^{#1}( #2 )}

\newcommand{\be}{\begin{equation}}
\newcommand{\ee}{\end{equation}}

\newcommand{\eps}{{\varepsilon}}
\newcommand{\set}[2]{\{#1\,:\,#2\}}

\newcommand{\C}{\mathcal C}
\newcommand{\cD}{\mathcal D}
\newcommand{\kD}{\mathfrak D}

\newcommand{\kT}{\mathfrak T}

\newcommand{\dist}{\mathrm{dist}}

\newcommand{\bnul}{{\boldsymbol 0}}

\newcommand{\br} {\bm r}

\newcommand{\bT} {\mathbf{T}}
\newcommand{\bmT} {\bm{T}}
\newcommand{\bF} {\mathbf{F}}

\newcommand{\D}{\mathrm{D}}

\newcommand{\dd}{\,{\rm d}}
\newcommand{\rd}{{\rm d}}
\newcommand{\ddx}{\dd\bx}

\newcommand{\dup}[2]{\langle #1, #2\rangle}

\renewcommand{\div}{{\rm div}}

\newcommand{\IC}{{\mathbb C}}

\newcommand{\IN}{{\mathbb N}}

\newcommand{\IR}{{\mathbb R}}

\newcommand{\half}{\frac{1}{2}}

\newcommand{\bp}{{\bm p}}

\newcommand{\K}{\mathfrak{K}}
\newcommand{\norm}[2]{\lVert #1\rVert_{#2}}
\newcommand{\tnorm}[1]{{\left\vert\kern-0.25ex\left\vert\kern-0.25ex\left\vert #1 
    \right\vert\kern-0.25ex\right\vert\kern-0.25ex\right\vert}}

\newcommand{\rK}{{\breve{K}}}

\newcommand{\Id}{\operatorname{\mathsf{I}}}

\renewcommand{\d}{\operatorname{d}}

\newcommand{\supp}{\operatorname{supp}}

\renewcommand{\div}{\operatorname{div}}

\newcommand{\rst}[1]{\left.#1\right|}  

\newcommand{\scurl}{\operatorname{curl}}

\newcommand{\bA}{\mathbf{A}}

\newcommand{\bE}{\mathbf{E}}
\newcommand{\bH}{\boldsymbol{H}}
\newcommand{\VH}{\mathbf{H}}
\newcommand{\bI}{\mathbf{I}}
\newcommand{\bn}{\mathbf{n}}

\newcommand{\bx}{\mathbf{x}}
\newcommand{\by}{\mathbf{y}}
\newcommand{\bJ}{\mathbf{J}}
\newcommand{\bW}{\boldsymbol{W}}
\newcommand{\bV}{\mathbf{V}}
\newcommand{\bU}{\mathbf{U}}

\newcommand{\tD}{\gamma_{\mathrm{D}}}
\newcommand{\tN}{\gamma_{\mathrm{N}}}

\numberwithin{equation}{section}

\crefname{hypothesis}{Hypothesis}{Hypotheses}
\crefname{assumption}{Assumption}{Assumptions}
\crefname{problem}{Problem}{Problems}

\title{Finite-Element Domain Approximation for Maxwell Variational Problems on Curved Domains}

\usepackage{amsopn}

\usepackage{authblk}

\ifpdf
\hypersetup{
  pdftitle={Finite-Element Domain Approximation for Maxwell Variational Problems on Curved Domains},
  pdfauthor={Rub\'en Aylwin and Carlos Jerez-Hanckes}
}
\fi

\begin{document}

\author{Rub\'en Aylwin}
\author{Carlos Jerez-Hanckes}
{\tiny 
\affil{Faculty of Engineering and Sciences, Universidad Adolfo Ib\'a\~nez, Santiago, Chile.}
}

\maketitle


%

\section{Introduction}
%
\label{sec:Intro}
We consider the problem of domain approximation in finite element methods for
Maxwell's equations on curved domains, i.e., when affine or polynomial meshes fail to cover
the domain of interest exactly, forcing to approximate the domain by a sequence
of (potentially curved) polyhedra arising from inexact meshes. In particular, we aim at finding conditions
on the quality of these approximations that ensure convergence rates of the
discrete solutions---in the approximate domains---to the continuous one in the original domain. This analysis is classical in the context of the Laplace equation \cite{ciarlet1972combined} but has not been studied in the Maxwell case. In \cite{aylwin2020effect}, we showed the effects of numerical
integration on the convergence of the curl-conforming finite element method (FEM) 
for Maxwell variational problems and found necessary conditions
on quadrature rules to ensure error convergence rates both in affine and curved meshes. However, we discarded the
error terms associated with domain mesh approximation. 

When approximating solutions to variational problems on a given \emph{original domain}
by solutions to analogous problems on 
\emph{approximate} (computational) domains, the choice of error measure is not straightforward as approximate 
and exact solutions do not share the same domain. Indeed,
several choices for error measures can be considered: comparisons between extensions of continuous or discrete solutions to a {\em hold-all} domain \cite{bhattacharyya1999combined,bramble1994robust,ciarlet1972combined,vanmaele1995combined}; mismatch measured at the intersection between the original and computational
domains \cite{lee2003curved,strang1973change,hernandez2003finite,thomee1973,Zl_mal_1974}; mapping of solutions from computational to
the original domains or vice-versa \cite{arnold2020hellan,di2018third,lenoir1986}; and finally, in \cite{chandler2021boundary,daners2003dirichlet}
the error is measured in a Hilbert space common to solutions on
the approximate and original domains.


In the present note, our main results are condensed in \cref{thm:pull_result_cont,thm:first_result_cont,thm:pull_result_disc,thm:first_result_disc}
in which we estimate the convergence of Maxwell solutions in a series of approximate domains $\{\widetilde\D_i\}_{i\in\IN}$
to the continuous solution in a given original domain---denoted $\D$, and approximated by the sequence $\{\widetilde\D_i\}_{i\in\IN}$---in two different ways. Following \cite{lenoir1986}, \cref{thm:pull_result_cont} estimates the error through curl-conforming \emph{pull-backs} mapping fields in approximate domains $\{\widetilde\D_i\}_{i\in\IN}$
to fields in the original one $\D$ (\emph{cf.}~\cite[Sec.~2.5]{jerez2017electromagnetic}). Alternatively, and in the spirit of \cite{ciarlet1972combined}, \cref{thm:first_result_cont} bounds the error of approximate
solutions to an extension of the solution in the original domain $\D$, allowing for its evaluation in each approximate domain
in the sequence even though one can not ensure that $\widetilde\D_i\subseteq\D$ for any $i\in\IN$ without additional assumptions.
Then, \cref{thm:pull_result_disc,thm:first_result_disc} correspond to discrete analogues to \cref{thm:pull_result_cont,thm:first_result_cont}, respectively.
Moreover, our findings allow for a straightforward combination with our earlier results
in \cite{aylwin2020effect}, and so \cref{thm:pull_result_disc_num,thm:first_result_disc_num}
correspond to fully discrete versions of \cref{thm:pull_result_disc,thm:first_result_disc}, respectively,
by incorporating the effects of numerical integration on error convergence rates.

The structure of the manuscript is as follows. In 
\cref{sec:GeneralDef} we set notation and introduce the Maxwell variational problems
considered throughout, as well as basic parameter and overarching assumptions.
In \cref{sec:FE}, we introduce finite elements on curved meshes
as in \cite{ciarlet1972combined} and introduce elementary results
concerning the continuity and approximation properties of the classical curl-conforming
interpolation operator (see \cite[Sec.~5.5]{Monk:2003aa}). \Cref{sec:ContProb}
introduces the issue of solving
Maxwell variational problems on approximate domains
at the continuous level; a viewpoint which is then directly applied
to the discrete level in \cref{sec:DiscProb}. Then, \cref{sec:NumRes} displays
a simple numerical example confirming our findings followed by
concluding remarks in Section \ref{sec:Conc}. Appendices provide proofs of various technical lemmas and results.

\section{General definitions and Maxwell problem statement}
\label{sec:GeneralDef}
\subsection{General notation}
Set $\imath=\sqrt{-1}$. For $d\in\IN$, we denote the canonical vectors in $\IR^d$ as
$\{\bm{e}_i\}_{i=1}^d$ and the inner product between
two elements $\bx$ and $\by$ in $\IR^d$ is written $\bx\cdot\by$.
Let $\Omega$ be an open bounded
Lipschitz domain in $\IR^d$ with boundary $\partial\Omega$. For $m\in\IN_0:=\IN\cup\{0\}$, $\C^m(\Omega)$
denotes the set of complex-valued functions with $m$-continuous
derivatives on $\Omega$, while $\C^m_0(\Omega)$ is the subset
of elements in $\C^m(\Omega)$ with compact support in $\Omega$.
Infinitely smooth functions with compact support in $\Omega$
belong to $\cD(\Omega):=\bigcap_{m=0}^\infty\C^m_0(\Omega)$.
For $k\in\IN_0$ and
$q\in\IN$, $\mathbb{P}_k(\Omega;\IC^q)$ is the space of polynomials of
degree less than or equal to $k$ from $\Omega$ to $\IC^q$.
$\widetilde{\mathbb{P}}_{k}(\Omega;\IC^q)$ denotes the space of homogeneous
polynomials of degree $k$ from $\Omega$ to $\IC^q$.
For $p\geq 1$ and $s\in\IR$, $L^p(\Omega)$ and $W^{p,s}(\Omega)$
are the class of $p$-integrable functions on $\Omega$ and 
the standard Sobolev spaces of order $s$, respectively.
If $p=2$, we employ the standard notation $H^s(\Omega):=W^{2,s}(\Omega)$.

Norms and semi-norms over a general Banach space $Y$ are
indicated by subscripts. However, the norm and semi-norm of $H^s(\Omega)$
will be written as $\norm{\cdot}{s,\Omega}$ and $\seminorm{\cdot}{s,\Omega}$,
respectively. The topological dual of the Banach space $Y$ will be denoted as $Y'$.
For a Hilbert space $X$, we write its inner product as
$\pr{\cdot,\cdot}_X$, and its duality pairing as
$\p{\cdot,\cdot}_{X'\times X}$. Again, we make an exception for
$H^s(\Omega)$ and write its inner and
duality products as $\pr{\cdot,\cdot}_{s,\Omega}$, and
$\p{\cdot,\cdot}_{s,\Omega}$, respectively. These
are understood in the sesquilinear sense.

General scalar-valued functions and function spaces are differentiated
from their vector-valued counterparts by the use of boldface symbols for the latter. Components of vector-valued functions are identified by subscript, e.g.,
$V_2=\bm{V}\cdot\bm{e}_2$.
For a square matrix $\bA\in\IC^{n\times n}$, with $n\in\IN$, we denote its
induced matrix norm by $\norm{\bA}{\IC^{n\times n}}$, its determinant by
$\det{\bA}$, its transpose by $\bA^\top$, its cofactor
matrix by $\bA^\co$ and its inverse by $\bA^{-1}={\det{\bA}}^{-1}\bA^\co$,
when invertible.
The Jacobian matrix of a differentiable function $\bU:\IR^n\to\IC^n$ is
$\rd\bU:\IR^n\to\IC^{n\times n}$. Moreover, $\Id:\IC^n\to\IC^n$ is the
identity map while $\bI\in\IR^{n\times n}$ denotes the identity matrix,
so that $\rd\Id=\bI$.

Finally, norms of vector-valued functions in $\bm{W}^{p,s}(\Omega)$, for $p\geq 1$ and $s\in\IR$,
are computed as the $p$-sum of the ${W}^{p,s}(\Omega)$-norms of their components, e.g.,~$\norm{\bU}{\bm{W}^{p,s}(\Omega)}^p
=\sum_{i=1}^d\norm{U_i}{{W}^{p,s}(\Omega)}^p$ for $p\in[1,\infty)$ and the customary modification when $p=\infty$. Norms for matrix-valued functions are computed analogously.
For a multi-index $\bm\alpha=(\alpha_1,\hdots,\alpha_d)^\top\in\IN_0^d$,
we write $\modulo{\bm\alpha}=\sum_{i=1}^d\alpha_i$ and
$\bx^{\bm\alpha}=\prod_{i=1}^n\bx_i^{\alpha_i}$. For $n\in\IN$, we
write the set of integers $\{1,2,3,\hdots,n\}$ as $\{1:n\}$
\subsection{Functional spaces}
\label{sec:funspaces}
Let $\Omega$ be an open and bounded Lipschitz domain in $\IR^3$. We introduce the following functional spaces of vector-valued
functions:
\begin{gather*}
\hcurl{\Omega}:=\left\{\bU\in \Lp{2}{\Omega} \ :\ \curl\bU\in \Lp{2}{\Omega} \right\},\\
\hcurlcurl{\Omega}{}:=\left\{\bU\in \hcurl{\Omega} \ :\ \curl\curl\bU\in \Lp{2}{\Omega} \right\},
\end{gather*}
together with the inner product on $\hcurl{\Omega}$:
\begin{gather*}
\pr{\bU,\bV}_{\hcurl{\Omega}}:=\pr{\bU,\bV}_{{0},{\Omega}}+\pr{\curl\bU,\curl\bV}_{{0},{\Omega}},
\end{gather*}
so that $\hcurl{\Omega}$ is Hilbert \cite[Sec.~3.5.3]{Monk:2003aa}.
For $s>0$, let us define the scale of
smooth spaces \cite[Sec.~3.5.3]{Monk:2003aa} useful to
characterize the regularity of Maxwell solutions:
\begin{align*}
\hscurl{\Omega}{s}:=\left\{\bU\in \Hsob{s}{\Omega} \ :\ \curl\bU\in \Hsob{s}{\Omega} \right\},
\end{align*}
with norm and semi-norm given by
\begin{align*}
\norm{\bU}{\bH^s(\curl;\Omega)}:=
\left(\norm{\curl\bU}{s,\Omega}^2+\norm{\bU}{s,\Omega}^2\right)^{\frac{1}{2}},\quad \seminorm{\bU}{\bH^s(\curl;\Omega)}:=
\left(\seminorm{\curl\bU}{s,\Omega}^2+\seminorm{\bU}{s,\Omega}^2\right)^{\frac{1}{2}}.
\end{align*}
We also require appropriate trace spaces
\cite{BuCo00,BufHipTvPCS_NM2003,Monk:2003aa}.
As in \cite{BuCo00}, we introduce two Hilbert
spaces of tangential vector fields on $\partial\Omega$ and their
duals: 
\begin{gather*}
\bH^{\half}_{\parallel}(\partial\Omega):=\{\bn\times(\bU\times\bn)\ :\ \bU\in\bH^{\half}(\partial\Omega)\},\quad
\bH^{\half}_{\perp}(\partial\Omega):=\{\bU\times\bn\ :\ \bU\in\bH^{\half}(\partial\Omega)\},\\
\bH^{-\half}_{\parallel}(\partial\Omega):=\left(\bH^{\half}_{\parallel}(\partial\Omega)\right)'\quad\mbox{and}\quad\bH^{-\half}_{\perp}(\partial\Omega):=\left(\bH^{\half}_{\perp}(\partial\Omega)\right)'.
\end{gather*}
where $\bn$ is the outward unit normal vector on
$\partial\Omega$. Trace spaces on $\hcurl{\Omega}$ are then defined through
first order differential operators on $\partial\Omega$:
\begin{align*}
\bH^{-\half}_{\div}(\partial\Omega) := \{\bU\in
\bH_\parallel^{-\half}(\partial\Omega):\; \div_{\partial\Omega}\bU \in
H^{-\half}(\partial\Omega) \},\\
\bH^{-\half}_{\scurl}(\partial\Omega) := \{\bU\in
\bH_\perp^{-\half}(\partial\Omega):\; \scurl_{\partial\Omega}\bU
\in H^{-\half}(\partial\Omega) \},
\end{align*}
where $\div_{\partial\Omega}$ and $\scurl_{\partial\Omega}$
are the divergence and scalar curl surface operators,
respectively (\emph{cf.}~\cite{BuHip} and \cite[Sec.~2.5.6]{Nedelec_2001} for detailed definitions).
Moreover, it holds that (\emph{cf.}~\cite[Thm.~2]{BuHip})
\begin{align*}
\bH^{-\half}_{\scurl}(\partial\Omega)=
\left(\bH^{-\half}_{\div}(\partial\Omega)\right)'.
\end{align*}
We define the following trace operators
\begin{gather*}
\tD:\bH(\curl;\Omega)\!\to\!\bH^{-\half}({\partial\Omega}),\quad
\tD^\times:\bH(\curl;\Omega)\!\to\!\bH^{-\half}({\partial\Omega}),\\ 
\tN:\hcurlcurl{\Omega}{}\!\to\! \bH^{-\half}_{\div}(\partial\Omega),
\end{gather*}
as the unique continuous extensions of their actions on
$\bU\in\bm{\C}^{\infty}(\overline{\Omega})$ given by
\begin{gather*}
\tD\bU := \bn \times(\rst{\bU}_{{\partial \Omega}}\times \bn),\quad 
\tD^\times\bU := \bn\times\rst{\bU}_{{\partial \Omega}}\quad \mbox{and}\quad
\tN\bU := \bn\times
\rst{\curl\bU}_{{\partial \Omega}},
\end{gather*}
dubbed the Dirichlet, flipped Dirichlet trace
and Neumann traces, respectively. Range spaces are characterized as
\begin{gather*}
\mathrm{Im}(\tD)=\bH^{-\half}_{\scurl}({\partial \Omega}),\qquad
\mathrm{Im}(\tD^{\times})=\bH^{-\half}_{\div}(\partial\Omega).
\end{gather*}
Moreover, for $\bU$ and
$\bV\in \bH(\curl,\Omega)$, the following Green identity holds
\begin{equation*}
(\bU,\curl\bV)_\Omega-(\curl\bU,\bV)_\Omega =
-\dup{\tD^\times\bU}{\tD\bV}_{\partial\Omega},\;
\end{equation*}
where $\dup{\cdot}{\cdot}_{\partial\Omega}$ denotes the
duality between $\bH_{\mathrm{div}}^{-\half}(\partial \Omega)$
and $\bH^{-\half}_{\scurl}(\partial\Omega)$ (\emph{cf.}~\cite[Sec.~3]{Monk:2003aa} and \cite{BuCo00}).

The subset of $\hcurl{\Omega}$-elements satisfying zero boundary conditions
is defined through the flipped Dirichlet trace as
\begin{gather*}
\bH_0(\curl;\Omega) := \{\bU\in \bH(\curl;\Omega) \ : \ \tD^\times \bU =
\bnul \; \mbox{on} \; \partial \Omega \}.
\end{gather*}
By continuity of the flipped Dirichlet trace, $\hocurl{\Omega}$
is a closed subspace of $\hcurl{\Omega}$.

\subsection{Maxwell variational problems}
Let $\D\subset\IR^3$ be an open, bounded domain
with boundary $\Gamma:=\partial\D$ of class $\C^{\mathfrak{M}}$ for some
$\mathfrak{M}\in\IN$. For a circular frequency $\omega > 0$ and time
dependency $e^{\imath \omega t}$, the time-harmonic Maxwell equations on $\D$ read
\begin{align}\label{eq:Maxwell}
\begin{aligned}
\curl \bE + \imath\omega \mu \VH &= \bnul,\\
\imath\omega\eps  \bE -\curl \VH &= -\bJ,
\end{aligned}
\end{align}
where $\bE$ and $\VH$ belong to $\hcurl{\D}$ and represent the electric and magnetic fields, respectively. The magnetic permeability $\mu$ and electric permittivity $\eps$
are assumed to be symmetric matrix-valued functions with coefficients in
$\lp{\infty}{\D}$, and $\bJ\in\Lp{2}{\D}$
is an imposed current in $\D$.

The system \cref{eq:Maxwell} is converted into a second order system for
$\bE$ or $\VH$ by eliminating the remaining field, requiring pointwise invertibility assumptions on either $\eps$ or $\mu$ depending on the specific
choice. Without loss of generality, we follow \cite{aylwin2020effect} and consider the system for
the electric field only, assuming the existence of a pointwise inverse of $\mu$. Thus, 
\begin{align*}
\VH=\imath\frac{1}{\omega}\mu^{-1}\curl\bE,
\end{align*}
from where
\begin{align}\label{eq:MaxE}
\curl \mu^{-1}\curl \bE - \omega^2\eps \bE = -\imath\omega\bJ.
\end{align}
The system is completed by imposing boundary
conditions on traces of $\bE$, e.g.,
\begin{align*}
\tD^{\times}\bE=\bm{g}_{D},\quad\mbox{or}\quad \bn\times(\mu^{-1}\curl\bE)=\bm{g}_N,
\end{align*}
for $\bm{g}_D\in \bH^{-\half}_{\div}(\partial\Omega)$ or
$\bm{g}_N\in\bH^{-\half}_{\curl}(\partial\Omega)$.

We proceed by considering the system \cref{eq:MaxE} with
perfect electric conductor (PEC) boundary conditions, i.e., homogeneous (flipped) Dirichlet boundary conditions given by $\tD^{\times}\bE=\bnul$.
The associated sesquilinear and antilinear forms on $\hocurl{\D}$
for the Maxwell PEC cavity problem, respectively, are 
\begin{align}
\label{eq:Phi}
\Phi(\bU,\bV)&:=
\int_{\D} \mu^{-1} \curl \mathbf{U} \cdot \curl \overline{\mathbf{V}}-
\omega^2\epsilon \mathbf{U} \cdot \overline{\mathbf{V}} \ddx\quad\mbox{and}\quad
\bF(\bV):=-\imath\omega\int_{\D}\bJ\cdot\overline{\bV}\ddx,
\end{align}
which are continuous on $\bH_0(\curl;\D)$ if we assume
$\mu^{-1}$ has coefficients in $\lp{\infty}{\D}$. Then, the problem under consideration reads:
\begin{problem}[Continuous variational problem]
\label{prob:varprob}%
Find $\bE\in \bH_0(\curl;\D)$ such that
\begin{align*}
\Phi(\bE,\bV)=\bF(\bV),
\end{align*}
for all $\bV\in \bH_0(\curl;\D)$.
\end{problem}
In this work, we are concerned with the approximation of
$\D$, the \emph{original domain}, by computational domains $\{\widetilde\D_i\}_{i\in\IN}$
and its consequences on the FEM error convergence rates.
Hence, we take for granted the necessary conditions for
the unique solvability of \cref{prob:varprob}.

\begin{assumption}[Wellposedness]
\label{ass:sesqform}
We assume the sesquilinear form $\Phi$ in \eqref{eq:Phi}
satisfies the following conditions:
\begin{align*}
\modulo{{\Phi}(\bU,\bV)}<C_1\norm{\bU}{\hcurl{\D}}\norm{\bV}{\hcurl{\D}}&\qquad\forall\;\bU,\;\bV\in\hocurl{{\D}},\\
\sup_{\bU\in \hocurl{\D}\setminus\{\bnul\}}\modulo{\Phi(\bU,\bV)}>0&\qquad\forall\;\bV\in \hocurl{\D}\setminus\{\bnul\},
\end{align*}
and 
$$\inf_{\bU\in \hocurl{\D}\setminus\{\bnul\}}\left(\sup_{\bV\in \hocurl{\D}\setminus\{\bnul\}} 
\frac{\modulo{\Phi(\bU,\bV)}}{\norm{\bU}{\hcurl{\D}}\norm{\bV}{\hcurl{\D}}}\right)\geq C_2,$$
for positive constants $C_1$ and $C_2$.
\end{assumption}
We denote the unique solution of \cref{prob:varprob} as
$\bE\in\hocurl{\D}$. Lastly, for examples of problems satisfying
\cref{ass:sesqform} we refer to \cite{AJZS18,ern2018analysis,Monk:2001aa,Monk:2003aa}.

\section{Curl-conforming finite elements}\label{sec:FE}
We begin by introducing the reference tetrahedron from which all
meshes will be constructed.

\begin{definition}[Reference element]
We define $\breve{K}$ as the tetrahedron with vertices
$\bnul$, $\bm{e}_1$, $\bm{e}_2$ and $\bm{e}_3$, and
refer to it as the reference element or reference tetrahedron.
\end{definition}

We also recall our smoothness assumptions on our original domain---stated in
\cref{sec:Intro}---, requiring $\D$ to be of class $\C^\mathfrak{M}$ for $\mathfrak{M}\in\IN$.

\begin{assumption}\label{ass:dom}
The bounded domain $\D$ is of class $\C^\mathfrak{M}$ for $\mathfrak{M}\in\IN$.
\end{assumption}

\Cref{ass:dom} is required to ensure convergence rates of
approximate domains to $\D$ built by polynomial interpolation
\cite{lenoir1986}. We point out that one could easily adjust the
following analysis to piecewise smooth domains.
\subsection{Curl-conforming finite element spaces on straight and curved meshes}
As in \cite{aylwin2020effect}, we introduce\footnote{The superindex $\pl$ stands for polyhedral.} $\kT^{\pl}$ a family of 
quasi-uniform straight meshes of $\D$, written $\tau^{\pl}_{h_i}$, with $h_i>0$ for all $i\in\IN$  $h_i\rightarrow0$ as $i$ grows to infinity, constructed by
straight tetrahedrons and indexed by their mesh-sizes, i.e., $\kT^{\pl}:=\{\tau^{\pl}_{h_i}\}_{i\in\IN}.$ Throughout, $\tau_h^{\pl}$ denotes an arbitrary mesh in $\kT^{\pl}$.
An arbitrary tetrahedron in any of
the meshes of $\kT^\pl$ is denoted $K^\pl$, and we assume
each tetrahedron $K^\pl$ to be constructed from $\breve{K}$ by an affine mapping,
denoted $\bmT_{K^\pl}:\breve{K}\to K^\pl$.  The polyhedral domain covered by
$\tau_h^\pl$ is denoted $\D^{\pl}_{h}$ with boundary
$\Gamma^{\pl}_{h}:=\partial\D^{\pl}_{h}$.

Now, for each polyhedral mesh $\tau_h^{\pl}\in\kT^{\pl}$,
we introduce $\tau_h$ as the \emph{approximated curved mesh} constructed
from $\tau_h^\pl$, in the sense that it shares its nodes with $\tau_h^{\pl}$
but is composed of curved tetrahedrons. As before, we introduce the family
of curved meshes as $\kT:=\{\tau_{h_i}\}_{i\in\IN}.$

For a given $K^\pl\in\tau_h^\pl$ we refer
to the element of $\tau_h$ that shares its nodes with $K^\pl$
as $K$ and consider bijective mappings
$\bmT_{K}:\breve{K}\mapsto K$ to be polynomial of degree $\K\in\IN$,
with $\K<\mathfrak{M}$ and fixed throughout. Also, we refer to an arbitrary mesh in $\kT$ by $\tau_h$ and the domain covered by $\tau_h$ by $\D_h$ with boundary
$\Gamma_h:=\partial\D_h$.

\begin{assumption}[Assumptions on $\kT^{\pl}$ and $\kT$.]
\label{ass:mesh}
The meshes in $\kT^{\pl}$ are assumed to be a\-ffine, quasi-u\-ni\-form and such that their
boundary nodes are located on $\Gamma$ and the polyhedral domains
$\{\D^{\pl}_{h_i}\}_{i\in\IN}$ approximate $\D$.
The family of approximate meshes $\kT$ is assumed to be
$\K$-regular, i.e., for each $K\in\tau_h$, the mappings ${\bmT}_{K}$ are $\C^{\mathfrak{K}+1}$-diffeomorphisms
that belong to $\mathbb{P}_{\mathfrak{K}}(\rK;\IR^3)$ for some integer
$\K<\mathfrak{M}$, with $\mathfrak{M}$ as in \Cref{ass:dom}. Moreover, they satisfy
\begin{align}\label{eq:cond:dTK}
\sup_{\bx\in\breve{K}}\norm{\rd^{n}{\bmT}_{K}(\bx)}{}\leq C_{n}h^{n}
\quad\mbox{and}\quad\sup_{\bx\in{K}}\norm{\rd^{n}\left({\bmT}_{K}^{-1}\right)(\bx)}{}
\leq C_{-n}h^{-n}\quad \forall\; n\in\{1:\mathfrak{K}+1\},
\end{align}
where $C_n$ and $C_{-n}$ are positive constants independent from
the mesh-size for all $n\in\{1:\mathfrak{K}+1\}$. Therein, $\rd^{n}{\bmT}_K$
is the Fr\'echet derivative of order $n$ of $\bmT_K$ and
$\Vert{\rd^{n}\widetilde{\bmT}_K(\bx)}\Vert$ is the
induced norm, with functional spaces omitted for brevity, and the curved domains
$\{\D_{h_i}\}_{i\in\IN}$ approximate $\D$.
Furthermore, we assume that $\det{\rd \bmT_K(\bx)}>0$ for all $\bx\in\rK$
and that there exists some positive $\theta\in\IR$,
independent of $h>0$, such that for all $ K\in\tau_h$, it holds that 
\begin{align*}
\frac{1}{\theta}\leq{\frac{\det{\rd \bmT_{K}(\bx)}}{\det{\rd \bmT_{K}(\by)}}}\leq \theta\quad\forall\;\bx ,\by\in \rK.
\end{align*}
\end{assumption}

Our assumptions on $\kT^\pl$
follow from \cite{ciarlet1972combined} and are satisfied by constructions of curved meshes by polynomial approximations
of the domain $\D$ (\emph{cf.}~\cite{lenoir1986}). \Cref{fig:Example} displays a 2D example of our
setting.

\begin{figure}[t]
\centering
\begin{subfigure}{0.5\textwidth}
\centering
\includegraphics[scale=0.5]{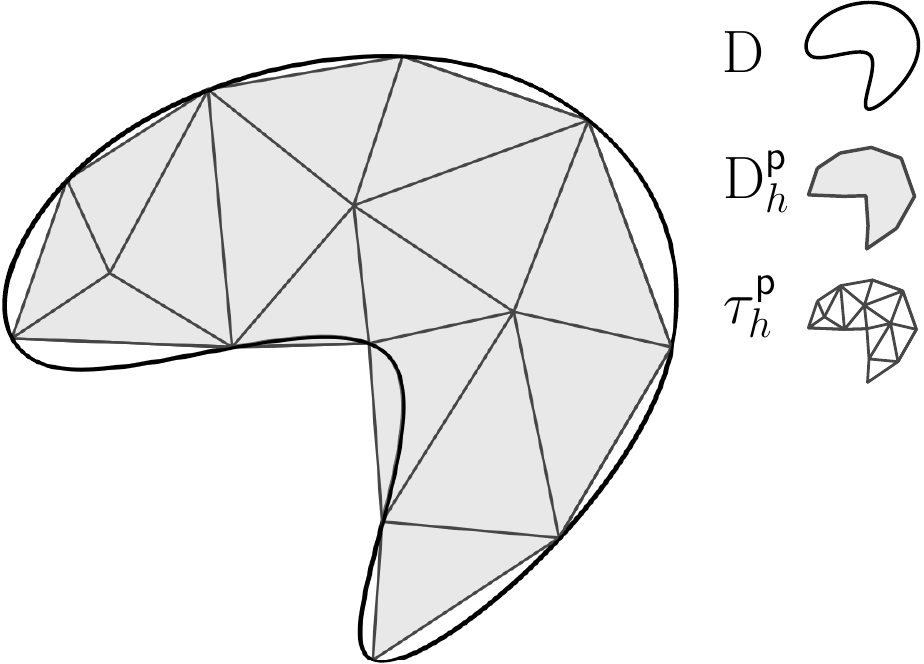}
\caption{Polygonal mesh of $\D$.}
\label{subfig:Domain_Example_Po}
\end{subfigure}%
\begin{subfigure}{0.5\textwidth}
\centering
\includegraphics[scale=0.5]{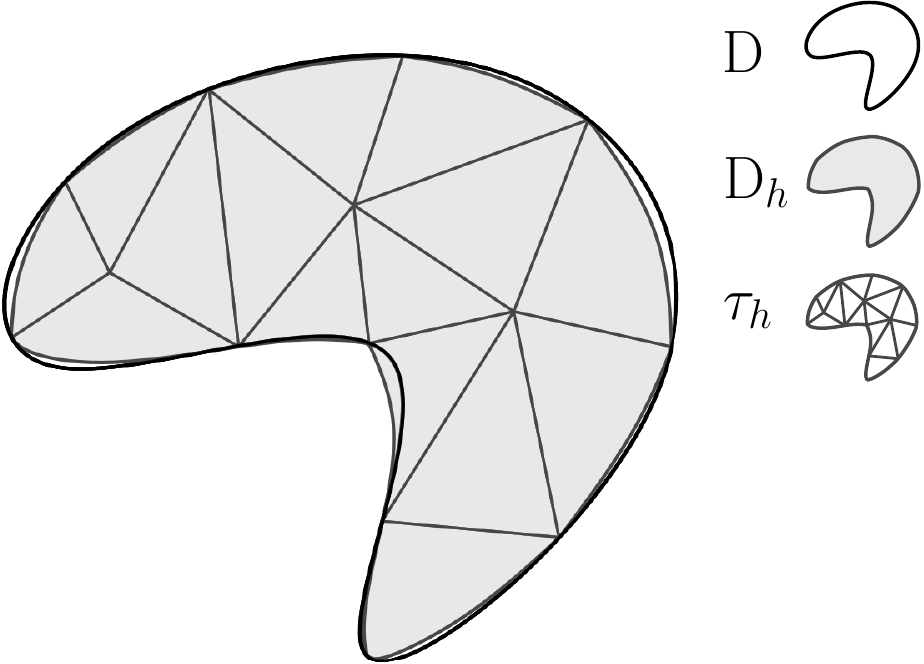}
\caption{Curved mesh of $\D$.}
\label{subfig:Domain_Example_Smo}
\end{subfigure}
\caption{Two-dimensional example of a smooth original domain $\D$, with associated straight (polygonal) and curved meshes, $\tau_h^\pl\in\kT^\pl_h$ of $\D$ (\cref{subfig:Domain_Example_Po}) and $\tau_h\in\kT$ of $\D$ (\cref{subfig:Domain_Example_Smo}), together with respective
approximated domains $\D^\pl_h$ and $\D_h$. Note that $\tau_h^\pl$ and $\tau_h$ share same nodes and that curved edges do not necessarily match the boundary of $\D$.}
\label{fig:Example}
\end{figure}

The sense in which we assume the approximate domains to converge to $\D$
will be made clear in the following section. Note that we have limited the
polynomial degree of our approximate domains by the smoothness of $\D$
(by imposing $\K<\mathfrak{M}$ where $\D$ is of class $\C^{\mathfrak{M}}$)
since no gain is derived from additional orders of approximation.
Moreover, for a multi-index ${\bm\alpha}\in\IN^3_0$, \Cref{ass:mesh} implies the following estimates:
\begin{gather}
\label{eq:normderbound}
\sup\limits_{\bx\in\breve K}\modulo{\frac{\partial^{\bm\alpha}}{\partial\bx^{\bm\alpha}}{{T}_{K,i}}(\bx)}\leq Ch^{\modulo{\bm\alpha}}\quad\mbox{and}\quad \sup\limits_{\bx\in K}\modulo{\frac{\partial^{\bm\alpha}}{\partial\bx^{\bm\alpha}}{{T}_{K,i}^{-1}(\bx)}}\leq Ch^{-\modulo{\bm\alpha}}\qquad\forall\; i\in\{1:3\},\\
\label{eq:bounddet}
ch^3 \leq \modulo{\det{\rd\bm T_K(\bx)}}\leq C  h^{3}\quad\forall\;\bx\in\breve{K},
\end{gather}
where $c$ and $C$ are positive generic constants---not necessarily equal in
each appearance---indepen-dent of $K$ and the mesh-size. The estimates
in \eqref{eq:normderbound} follow from norm equivalence over finite-dimensional spaces and \eqref{eq:cond:dTK}, while \eqref{eq:bounddet}
follows from \eqref{eq:normderbound} by straightforward computation
(\emph{cf.}~Lemma 8 in \cite{aylwin2020effect}).

With the above definitions, we consider finite elements as triples $(K,P_K,\Sigma_K)$,
with $K\in\tau_h$, $P_K$ a space of polynomials over $K$ and
$\Sigma_K:=\{\sigma^{K}_{i}\}_{i=1}^{n_\Sigma},\ n_\Sigma\in\IN,$ a set of linear
functionals acting on $P_K$ (\emph{cf.}~\cite{Monk:2003aa}).
Let $k\in\IN$ refer to the polynomial degree
of the curl-conforming (N\'ed\'elec)
finite element space on the reference tetrahedron $\breve{K}$
defined as
\begin{gather}\label{eq:fe_spaces}
\begin{gathered}
{\bm{P}}^c_{\breve{K}}:=\mathbb{P}_{k-1}(\breve{K};\IC^3)\oplus \lbrace\bm{p}\in 
\widetilde{\mathbb{P}}_{k}(\breve{K},\IC^3)\ : \bx\cdot\bp(\bx)=0\rbrace.
\end{gathered}
\end{gather}
Finite element spaces on arbitrary tetrahedrons $K$ (straight or curved)
are defined via a curl-conforming {\em pull-back} as follows
\begin{gather}
{\bm{P}}^c_{{K}}:=\{\bm{p}\; :\; \psi^{c}_K(\bm{p})\in\bm{P}^c_{\breve{K}}\}\quad\mbox{where}\quad
\psi^c_K(\bV):=\rd \bmT_K^\top (\bV\circ \bmT_K).
\label{eq:fe_spaces_K}
\end{gather}
The pull-back in \cref{eq:fe_spaces_K} defines an isomorphism between
$\hcurl{K}$ and $\hcurl{\rK}$ and satisfies
(\emph{cf.}~\cite[Lem.~2.2]{jerez2017electromagnetic} and \cite{Monk:2003aa})
\begin{align*}
\curl\psi^c_K(\bV)=\rd \bmT_K^\co\curl\bV\circ \bmT_K.
\end{align*}

We refer to \cite[Sec.~5.5]{Monk:2003aa} for the definition of degrees
of freedom for the reference finite element space in \cref{eq:fe_spaces}.The curl-conforming discrete spaces on $\tau_h\in\kT$ are then constructed as
\begin{gather}\label{eq:disc_space}
\begin{gathered}
\bm P^c(\tau_h):= \left\lbrace \bV_h\in\hocurl{\D_h} \ :\  
\bV_h\vert_K \in {\bm{P}}^c_K\quad \forall\; K\in\tau_h\right\rbrace,
\end{gathered}
\end{gather}
where $\D_h$ is the domain covered by $\tau_h$.

\subsection{Curl-conforming interpolation on curved meshes}

We now focus on proving continuity and approximation properties for the classical
curl-conforming interpolation operator on curved meshes.
We shall denote the reference curl-conforming finite element by $(\rK,\bm{P}_{\rK}^c,\Sigma_{\rK}^c)$
and let us introduce $\{\bm\phi_{\sigma}\}_{\sigma\in\Sigma_{\rK}^c}$ as the basis
of $\bm{P}_{\rK}^c$ associated with the degrees of freedom $\Sigma_{\rK}^c$ so that for any pair of
degrees of freedom $\sigma$, $\sigma'\in\Sigma_{\rK}^c$, it holds that
\begin{align*}
\sigma(\bm\phi_{\sigma'})=\begin{cases}
1,\quad \mbox{if }\sigma=\sigma',\\
0,\quad \mbox{if } \sigma\neq\sigma'.
\end{cases}
\end{align*}
For further details, we refer to \cite[Sec.~5.5]{Monk:2003aa}.

\begin{definition}[Local interpolation operator]
Let $s\in\IN$. We define the
canonical interpolation operator on $\breve{K}$
$$\breve{\br}:\hscurl{\breve K}{s}\to\bm{P}^c_{\breve K}$$ as
the operator mapping $\bU\in\hscurl{\breve K}{s}$ to the unique element
in $\bm{P}^c_{\breve K}$ having the same degrees of freedom as $\bU$, i.e.,
\begin{align*}
\breve\br(\bU):=\sum\limits_{\sigma\in\Sigma^c_\rK}\sigma(\bU)\bm{\phi}_\sigma.
\end{align*}

For any $\tau_h\in\kT$ and any $K\in\tau_h$ we denote the canonical interpolation
operator on $K$ as $\br_K$, mapping $\bU\in\hscurl{K}{s}$ to $\bm{P}^c_{K}$ as follows
\begin{align*}
\br_K:\left\lbrace\begin{aligned}\hscurl{K}{s}&\to\bm{P}^c_{ K},\\
\bU&\mapsto (\psi^{c}_K)^{-1}(\breve\br(\psi^c_K(\bU))).\end{aligned}\right.
\end{align*}
\end{definition}


\begin{assumption}\label{ass:k_leq_K}
From here onwards, we assume that $k\leq\K$.
\end{assumption}
The next results are proven in \cref{sec:tech_res2}.
\begin{proposition}\label{prop:locinterp}
Let \cref{ass:mesh,ass:k_leq_K} hold. For $K\in\tau_h$ and $\bU\in\hscurl{K}{s}$ for some $s\in\{1:k\}$, one has that
\begin{align*}
\norm{\bU-\br_K(\bU)}{\hcurl{K}}\leq C h^s\norm{\bU}{\hscurl{K}{s}},
\end{align*}
where $C>0$ is independent of $K$, $\bU$
and $\tau_h\in\kT$.
\end{proposition}

\begin{proposition}\label{prop:locinterp_cont}
Let \cref{ass:mesh,ass:k_leq_K} hold. For $K\in\tau_h$ and $\bU\in\hscurl{K}{s}$ with $s\in\{1:k\}$, one can show that
\begin{align*}
\norm{\br_K(\bU)}{\hscurl{K}{s}}&\leq c \norm{\bU}{\hscurl{K}{s}},
\end{align*}
where $c>0$ is independent of $K$, $\bU$
and $\tau_h\in\kT$.
\end{proposition}

\begin{definition}[Global interpolation operator]
\label{defi:glob_interp}
Let $s\in\IN$. For all $\tau_h\in\kT$ define the canonical interpolation operator on the entire mesh $\tau_h$ as
$$\mathbf{\Pi}_h:\hscurl{\D_h}{s}\to\bm{P}^c(\tau_h),$$
i.e., the operator mapping $\bU\in\hscurl{\D_h}{s}$
to the unique element in $\bm{P}^c(\tau_h)$ having the same degrees of freedom
as $\bU$ element by element:
\begin{align*}
\mathbf{\Pi}_h(\bU)|_{K}=\br_K(\bU)\quad\forall\, K\in\tau_h.
\end{align*}
\end{definition}
The global interpolation operator follows \cite[Sec.~5.5]{Monk:2003aa} and enjoys results analogous to
\cref{prop:locinterp,prop:locinterp_cont}, which we omit since we will require more specialized
versions later on.
\section{Variational problems on approximate domains:
continuous problem}
\label{sec:ContProb}
We now focus on the solution of \cref{prob:varprob}
on a countable family of domains $\kD:=\{\widetilde\D_i\}_{i\in\IN}$
that approximate the original domain $\D$.
Specifically, we are interested in
computing the rate of convergence of solutions
on each domain $\widetilde\D\in\kD$ to $\bE$ in $\D$.
Rather than immediately considering the
approximate domains defined by meshes in $\kT$, we study
the problem in a more general setting, so as to derive conditions
on $\kD$ transferable to our meshes in $\kT$.
We will return to our original discrete problem---identifying
$\widetilde\D_i$ with $\D_{h_i}$---in \Cref{sec:DiscProb}.
Moreover, since we are to consider \cref{prob:varprob} for domains in
$\kD$, which need not be contained in $\D$, we require the
data $\mu$, $\eps$ and $\bJ$ of \cref{prob:varprob} to
have extensions to a \emph{hold-all domain}, denoted $\D_H$,
containing $\D$ and each domain in $\kD$.

\begin{assumption}[Extension of parameters]\label{ass:extension}
There exists an open and bounded Lipschitz domain $\D_H$,
referred to as the {\em hold-all domain}, such that $\overline\D\subset\D_H$ and
$\overline{\widetilde\D}\subset\D_H$ for all $\widetilde\D\in\kD$.
Both $\mu$ and $\epsilon$ are complex symmetric matrix-valued functions with
coefficients in $\lp{\infty}{\D_H}$ and $\mu$ has a pointwise inverse ($\mu^{-1}$)
almost everywhere on $\D_H$, with coefficients in $\lp{\infty}{\D_H}$ as well.
The imposed current $\bJ$ may be extended to $\D_H$
so that $\bF$ in \eqref{eq:Phi} may be extended to $\hocurl{\D_H}'$.
\end{assumption}

We consider $\hocurl{\D}$ and $\hocurl{\widetilde\D}$
to be closed subspaces of $\hcurl{\D_H}$ by 
identifying elements in $\hocurl{\D}$ and $\hocurl{\widetilde\D}$
with their extension by $\bnul$ to $\D_H$. We may then continuously
extend the sesquilinear form in \cref{eq:Phi} as follows:
\begin{align}
\label{eq:Phi_approx}
\Phi(\bU,\bV)&:=
\int_{\D_H} \mu^{-1} \curl \mathbf{U} \cdot \curl \overline{\mathbf{V}}-
\omega^2\epsilon \mathbf{U} \cdot \overline{\mathbf{V}} \d\!\bx\qquad\forall\;\bU,\bV\in\hocurl{\D_H}
\end{align}
while the right-hand side $\bF$ in \cref{eq:Phi}
is extended to $\hocurl{\D_H}'$ by \Cref{ass:extension}, e.g., by
taking an $\Lp{2}{\D_H}$-extension of $\bJ$.
\begin{problem}[Continuous variational problem on inexact domains]
\label{prob:varprob_approx}%
Find $\widetilde\bE \in \bH_0(\curl;\widetilde\D)$ such that
\begin{align*}
\Phi(\widetilde\bE,\bV)=\bF(\bV),
\end{align*}
for all $\bV\in \bH_0(\curl;\widetilde\D)$.
\end{problem}
As before, we assume \Cref{prob:varprob_approx} is well
posed on each $\widetilde\D\in\kD$, with uniform constants.
\begin{assumption}[Wellposedness on $\kD$]\label{ass:sesqform_approx}
We assume the sesquilinear form in 
\cref{eq:Phi_approx} to satisfy the following conditions:
\begin{align*}
\modulo{\Phi(\bU,\bV)}<C_1\norm{\bU}{\hcurl{\widetilde\D}}\norm{\bV}{\hcurl{\widetilde\D}}&\quad\forall\;\bU,\;\bV\in\hocurl{\widetilde{\D}},\\
\sup_{\bU\in \hocurl{{\widetilde\D}}\setminus\{\bnul\}}\modulo{\Phi(\bU,\bV)}>0&\quad
\forall\;\bV\in \hocurl{{\widetilde\D}}\setminus\{\bnul\},
\end{align*}
and $$\inf_{\bU\in \hocurl{{\widetilde\D}}\setminus\{\bnul\}}\left(\sup_{\bV\in \hocurl{{\widetilde\D}}\setminus\{\bnul\}} 
\frac{\modulo{\Phi(\bU,\bV)}}{\norm{\bU}{\hcurl{{\widetilde\D}}}\norm{\bV}{\hcurl{{\widetilde\D}}}}\right)\geq C_2,$$
for all $\widetilde\D\in\kD$, with positive constants $C_1$ and $C_2$ independent
of $\widetilde\D\in\kD$.
\end{assumption}

\begin{example}
Taking $\mu^{-1}$ and $\epsilon$ in $\Lp{\infty}{\D_H;\IC^{3\times 3}}$ and such that
\begin{align*}
\inf _{\mathbf{x} \in D_{H}} \operatorname{Re}\left(e^{\imath \theta}\mu(\bx)^{-1}\right), \  \inf _{\mathbf{x} \in D_{H}} \operatorname{Re}\left(-e^{\imath \theta} \epsilon(\bx)\right)\geq\alpha>0
\end{align*}
for some $\theta\in[0,2\pi)$ and $\alpha>0$ is enough to ensure the conditions in \Cref{ass:sesqform_approx} (\emph{cf.}~\cite{AJZS18,ern2018analysis}).
\end{example}
For general $\widetilde\D\in\kD$, we denote the unique solution of
\cref{prob:varprob_approx} as $\widetilde\bE\in\hocurl{\widetilde\D}$, respectively, $\widetilde\bE_i\in\hocurl{\widetilde\D_i}$ for each $i\in\IN$.

\subsection{On the convergence of domains}
We introduce several different notions of
convergence of a sequence of domains to a limit, so that our 
conditions on the sequence $\kD$ are clearly defined.

\begin{definition}[Mosco convergence]\label{def:Mosco}
We say $\kD$ approximates $\D$ in the sense of Mosco
if the following conditions hold:
\begin{enumerate}[label=(\alph*)]
\item For every $\bU\in\hocurl{\D}$ there exists a sequence
$\{\bU_i\}_{i\in\IN}$, with
$\bU_i \in \hocurl{\widetilde\D_i}$ for all $i\in\IN$, such that
$\bU_i$ converges to $\bU$ strongly in $\hcurl{\D_H}$.\label{it:1def:Mosco}
\item Weak limits in $\hocurl{\D_H}$ of every sequence $\{\bU_i\}_{i\in\IN}$ satisfying
$\bU_i\in\hocurl{\widetilde\D_i}$ for all $i\in\IN$, belong to $\hocurl{\D}$.
\label{it:2def:Mosco}
\end{enumerate}
Note that we have identified each $\bU\in\hocurl{\widetilde\D}$ with its extension by zero to $\hocurl{\D_H}$.
\end{definition}

The notion of Mosco convergence originated in the study of variational
inequalities \cite{mosco1969convergence,mosco2011introduction} with applications to partial differential equations found in \cite{chandler2021boundary,daners2003dirichlet,Pironneau:1984}.
The original definition of Mosco convergence corresponds to the
convergence of the spaces $\hocurl{\widetilde\D_i}$ to $\hcurl{\D}$
rather than to the convergence of the domains $\widetilde\D_i$ to $\D$,
but we choose the latter convention since both are equivalent in our context.

\begin{lemma}\label{lem:mosco}
Let \cref{ass:sesqform,ass:sesqform_approx,ass:extension} hold and
let $\bE$ and $\widetilde\bE_i$ denote the unique solutions of
\cref{prob:varprob,prob:varprob_approx}, respectively.
Assume $\kD$ approximates $\D$ in the sense of Mosco. Then,
$\{\widetilde\bE_{i}\}_{i\in\IN}$ converges to $\bE$ in $\hcurl{D_H}$. 
\end{lemma}
\begin{proof}
Let $\{\bE_i\}_{i\in\IN}$ be a sequence as in
\cref{it:1def:Mosco} in \cref{def:Mosco}, strongly converging
to $\bE$ in $\hcurl{\D_H}$. Then, by
\Cref{ass:sesqform_approx}, for each $i\in\IN$ there exists
$\bV_i\in\hocurl{\widetilde\D_i}$, with
$\norm{\bV_i}{\hcurl{\widetilde\D_H}}=1$, such that
\begin{align}      
\frac{ C_2}{2}\norm{\widetilde\bE_{i}-\bE_{i}}{\hcurl{\D_H}}\leq{\modulo{\Phi(\widetilde\bE_{i}-\bE_{i},\bV_{i})}}
=\modulo{\bF(\bV_i)-\Phi(\bE_i,\bV_i)},
\label{eq:ineq_mosco}
\end{align}
where the positive constant $C_2$ is as in \cref{ass:sesqform_approx}.
Moreover, since the sequence $\{\bV_i\}_{i\in\IN}$ is bounded
it has a weakly convergent subsequence\textemdash still denoted
$\{\bV_i\}_{i\in\IN}$\textemdash to a limit point $\bV\in\hocurl{\D}$ due to
\cref{it:2def:Mosco} in \cref{def:Mosco}, so that 
\begin{align*}
\lim\limits_{i\rightarrow\infty}\bF(\bV_i)=\bF(\bV)\quad\mbox{and}\quad\lim\limits_{i\rightarrow\infty}\Phi(\bE_i,\bV_i)=\Phi(\bE,\bV),
\end{align*}
and the result follows by taking the limit as $i$ grows to infinity in \eqref{eq:ineq_mosco}.
\end{proof}



Notice that it is not straightforward to derive convergence rates of
approximate solutions $\widetilde\bE_i$ to $\bE$, since we cannot
estimate $\norm{\bE-\bE_{i}}{\hcurl{\D_H}}$ in
\cref{lem:mosco} without
further assumptions on $\kD$. However, the notion
of Mosco convergence gives minimum conditions to ensure strong
convergence of the approximate solutions.\footnote{The conditions in
\cref{def:Mosco} are further studied in \cite{daners2003dirichlet}
and Lemma 2.7 in \cite{chandler2021boundary}, for example.}

\begin{definition}[Hausdorff convergence]\label{def:Haus}
We say $\kD$ approximates $\D$ in the sense of Hausdorff if
\begin{align*}
\lim\limits_{i\rightarrow\infty}d_{\mathcal{H}}(\overline\D_H\setminus\widetilde\D_i,\overline\D_H\setminus\D)=0,
\end{align*}
where $d_{\mathcal{H}}(\cdot,\cdot)$ denotes the Hausdorff metric between closed subsets of
$\IR^3$, defined as
\begin{align*}
d_{\mathcal{H}}(\Omega_1,\Omega_2):=\max\left\{\sup\limits_{\bx\in\Omega_1}\dist(\bx,\Omega_2),\sup\limits_{\by\in\Omega_2}\dist(\by,\Omega_1)\right\},
\end{align*}
for two closed subsets $\Omega_1$ and $\Omega_2$ of $\IR^3$.
\end{definition}

\begin{lemma}[Lemmas 3 and 4 in \cite{Pironneau:1984}]
\label{lem:HausImpMosc}
Suppose $\kD$ approximates $\D$ in the sense of Hausdorff and that
$\D$ and all $\widetilde\D\in\kD$ are Lipschitz continuous domains
with uniform Lipschitz constant. Then, $\kD$ approximates $\D$ in the sense
of Mosco.
\end{lemma}


\cref{lem:HausImpMosc} shows that uniform point-wise approximation---together with mild assumptions on the regularity of $\D$ and $\kD$---implies
Mosco convergence, and so it provides us with sufficient---geometric---conditions that ensure the strong convergence of
$\{\widetilde\bE_i\}_{i\in\IN}$ to $\bE$. Still, the notion is too weak
for us to compute meaningful estimates as it gives almost no information
on the domains {$\widetilde\D$} in $\kD$. 
Instead of the previous definitions of convergence of domains,
we shall consider the following (stronger) notion, which appears in
\cite{DelfZol_2ndEd_2011,Sokolowski:1992} in the context of shape optimization.

\begin{definition}[Convergence in the sense of transformations]\label{def:Trans}
We say $\kD$ approximates $\D$ in the \emph{sense of transformations} of order
$\mathfrak{n}\in\IN_0$
if there exist bijective transformations $\{\bT_{i}\}_{i\in\IN}$ such that:
\begin{gather*}
\bT_i:\D_H\to\D_H,\quad\bT_i\vert_{\widetilde\D_i}:\widetilde\D_i\to\D,\quad\bT_i,\bT_i^{-1}\in \bm{W}^{\mathfrak{n},\infty}(\D_H)\\
\lim\limits_{i\rightarrow\infty}\norm{\bT_i-\Id}{\bm{W}^{\mathfrak{n},\infty}(\D_H)}+\norm{\bT_i^{-1}-\Id}{\bm{W}^{\mathfrak{n},\infty}(\D_H)} = 0.
\end{gather*}
For any transformation $\bT$ satisfying the previous conditions for a domain $\widetilde\D\in\kD$, we denote the associated
discrepancy between $\D$ and $\widetilde\D$, subject to the transformation $\bT$, as 
\begin{align*}
d_\mathfrak{n}(\D_H,\bT):=\norm{\bT-\Id}{\bm{W}^{\mathfrak{n},\infty}(\D_H)}+\norm{\bT^{-1}-\Id}{\bm{W}^{\mathfrak{n},\infty}(\D_H)}.
\end{align*}
\end{definition}

It is straightforward to see that convergence
in the sense of transformations of order zero implies Hausdorff convergence
and that convergence of order one implies, together with the Lipschitz
continuity of $\D$, the results of \cref{lem:HausImpMosc}. We continue our analysis under the following assumption.

\begin{assumption}[Assumptions on $\kD$]\label{ass:app_dom}
We assume that the countable family $\kD$ approximates $\D$ in
the sense of transformations of order $\mathfrak{n}=1$. Moreover,
we assume the respective family of transformations $\{\bT_{i}\}_{i\in\IN}$
is such that
\begin{gather}
d_1(\D_H,\bT_i)<1,\quad d_1(\D_H,\bT_{i+1})<d_1(\D_H,\bT_i),\label{eq:ass_trans_small}\\
\label{eq:vartheta_cond}
\begin{gathered}
 \vartheta^{-1}\leq
 \norm{\det{\rd\bT_i}}{\lp{\infty}{\D_H}},\;\norm{\rd\bT_i}{\Lp{\infty}{\D_H;\IC^{3\times 3}}},\;\norm{{\rd\bT_i}^\co}{\lp{\infty}{\D_H;\IC^{3\times 3}}},
 \leq\vartheta,\\
  \vartheta^{-1}\leq
 \norm{\det{\rd(\bT_i^{-1})}}{\lp{\infty}{\D_H}},\;\norm{\rd(\bT_i^{-1})}{\Lp{\infty}{\D_H;\IC^{3\times 3}}},\;\norm{{\rd(\bT_i^{-1})}^\co}{\lp{\infty}{\D_H;\IC^{3\times 3}}},
 \leq\vartheta,
 \end{gathered}
\end{gather}
for some $\vartheta>1$ and for all $i\in\IN$. 

We will further assume, for simplicity, that all determinants $\det{\rd\bT}$ are
positive almost everywhere on $\D_H$.

\end{assumption}

\begin{remark}
The conditions in \cref{eq:ass_trans_small,eq:vartheta_cond}
only restrict the quality of ``bad'' approximations of
$\D$, as convergence in the sense of transformations of order one implies
\begin{gather*}
\lim\limits_{i\rightarrow\infty} \norm{\det{\rd\bT_i}}{\lp{\infty}{\D_H}},\;\norm{\rd\bT_i}{\Lp{\infty}{\D_H;\IC^{3\times 3}}},\;\norm{{\rd\bT_i}^\co}{\lp{\infty}{\D_H;\IC^{3\times 3}}}
=1,\\
\lim\limits_{i\rightarrow\infty}\norm{\det{\rd(\bT_i^{-1})}}{\lp{\infty}{\D_H}},\;\norm{\rd(\bT_i^{-1})}{\Lp{\infty}{\D_H;\IC^{3\times 3}}},\;\norm{{\rd(\bT_i^{-1})}^\co}{\lp{\infty}{\D_H;\IC^{3\times 3}}}=1.
\end{gather*}
Moreover, by the norm equivalence over finite-dimensional spaces---and some algebra in the last case---
it holds that
\begin{align}
\norm{\rd\bT_i-\bI}{\Lp{\infty}{\D_H}},\;\norm{\rd\bT_i^{-1}-\bI}{\Lp{\infty}{\D_H;\IC^{3\times 3}}},\;\norm{\rd\bT_i^\co-\bI}{\Lp{\infty}{\D_H;\IC^{3\times 3}}}\leq Cd_1(\D_H,\bT_i),\label{eq:CboundMat}
\end{align}
where $C>0$ is independent of $i\in\IN$.
\end{remark}

As aforementioned, we will assess the quality of the
solutions of \cref{prob:varprob_approx} as approximations to the solution of \cref{prob:varprob} in two different ways:
\begin{enumerate}[label=(\alph*)]
\item through isomorphisms $\Psi_i:\hocurl{\D}\to\hocurl{\widetilde\D_i}$ to measure $\norm{\Psi_i\bE-\widetilde\bE_i}{\hcurl{\widetilde\D_i}}$ as $i$ grows towards infinity---equivalently, $\norm{\bE-\Psi_i^{-1}\widetilde\bE_i}{\hcurl{\D}}$---; and,
\item through an appropriate extension to $\D_H$ of $\bE$, allowing us
to measure $\norm{\bE-\widetilde\bE_i}{\hcurl{\widetilde\D_i}}$ as $i$
grows towards infinity.
\end{enumerate}

We now introduce a curl-conforming pull-back
that will act as the mentioned isomorphism between $\hocurl{\widetilde\D}$ and $\hocurl{\D}$.

\begin{lemma}[Lemma 2.2 in \cite{jerez2017electromagnetic}]\label{lem:trans_curl}
For $\widetilde\D\in\kD$, let $\bT:\widetilde{\D}\to\D$ be a continuous, bijective
and bi-Lipschitz mapping from $\widetilde\D$ to $\D$, so that $\bT\in\bW^{1,\infty}(\widetilde\D)$ and
$\bT^{-1}\in\bW^{1,\infty}(\D)$. Then, $\bT$ induces
an isomorphism between $\hocurl{\D}$ and $\hocurl{\widetilde\D}$, given by 
\begin{gather*}
{\Psi}:\left\{\begin{aligned}\hocurl{\D}&\to\hocurl{\widetilde\D}\\
\bU&\mapsto\rd\bT^{\top}(\bU\circ\bT)\end{aligned}\right. .
\end{gather*}
Moreover, it holds that
\begin{gather*}
\curl\Psi(\bU)=\rd\bT^{\co}(\curl\bU\circ\bT)\in\Lp{2}{\widetilde\D}.
\end{gather*}
\end{lemma}
Since $\bT$ and $\bT^{-1}$ possess analogous properties, the results of
\cref{lem:trans_curl} hold for $\bT^{-1}$ as well. Hence, we shall denote the inverse of
the mapping $\Psi:\hocurl{\D}\to\hocurl{\widetilde\D}$ by $\Psi^{-1}$, for which one has
\begin{gather*}
\begin{gathered}
{\Psi}^{-1}:\left\{\begin{aligned}\hocurl{\widetilde\D}&\to\hocurl{\D}\\
\bU&\mapsto\rd(\bT^{-1})^{\top}(\bU\circ\bT^{-1})\end{aligned}\right. ,\\
\curl\Psi^{-1}(\bU)=\rd(\bT^{-1})^{\co}(\curl\bU\circ\bT^{-1})\in\Lp{2}{\widetilde\D}.
\end{gathered}
\end{gather*}

From here onwards, and for general $\widetilde\D\in\kD$, we refer to the isomorphism introduced in \cref{lem:trans_curl} as $\Psi:\hocurl{\D}\to\hocurl{\widetilde\D}$---respectively, $\Psi_i:\hocurl{\D}\to\hocurl{\widetilde\D_i}$ for each $i\in\IN$.

\begin{lemma}\label{lem:curl_pull_cont}
Let \cref{ass:app_dom} hold.
Then, 1the following bounds are satisfied
\begin{align*}
\norm{\Psi_i\bU}{\hcurl{\widetilde\D_i}}\leq C\norm{\bU}{\hcurl{\D}}\quad\mbox{and}\quad\norm{\Psi_i^{-1}\bV}{\hcurl{\D}}\leq C\norm{\bV}{\hcurl{\widetilde\D_i}},
\end{align*}
for all $\bU\in\hocurl{\D}$ and all $\bV\in\hocurl{\widetilde\D_i}$, where the positive constant $C$ depends on $\vartheta>1$ introduced in \cref{ass:app_dom}, but not on $i\in\IN$.
\end{lemma}
\begin{proof}
Fix $i\in\IN$ and let $\bU\in\hocurl{\D}$. Then, by
\cref{ass:app_dom}, one has
\begin{align*}
\norm{\Psi_i\bU}{0,{\widetilde\D_i}}^2&=\int_{\widetilde\D_i}\norm{\Psi_i\bU(\bx)}{\IC^3}^2\d\!\bx=\int_{\widetilde\D_i}\norm{\rd\bT_i^\top\bU\circ\bT(\bx)}{\IC^3}^2\d\!\bx\\
&\leq\int_{\widetilde\D_i}\norm{\rd\bT_i^\top(\bx)}{\IC^{3\times 3}}^2\norm{\bU\circ\bT(\bx)}{\IC^3}^2\d\!\bx\leq C\vartheta^2\int_{\widetilde\D_i}\norm{\bU\circ\bT(\bx)}{\IC^3}^2\d\!\bx\\
&=C\vartheta^2\int_{\D}\norm{\bU(\bx)}{\IC^3}^2\;\det{\rd(\bT_i^{-1}(\bx))}\d\!\bx\leq C\vartheta^3\norm{\bU}{0,{\D}}^2,
\end{align*}
where the positive constant $C$ follows from the norm equivalence over finite dimensional spaces.
An analogous computation yields
\begin{align*}
\norm{\curl\Psi_i\bU}{0,{\widetilde\D_i}}^2\leq C\vartheta^3\norm{\curl\bU}{0,{\D}}^2.
\end{align*}
From where the estimate for $\Psi_i$ follows straightforwardly. The estimate for $\Psi_i^{-1}$ is retrieved by repeating
the arguments exposed above.
\end{proof}

The next results follow from arguments similar to
\cite[Prop.~2.32]{Sokolowski:1992} and will be of use throughout (\emph{cf.}~\cite[Lem.~5.1]{Babu_ka_2003}), and whose proofs are provided in \cref{sec:tech_res_T}.

\begin{lemma}\label{lem:error_infty}
Let $\Upsilon$ and $\Omega$ be open Lipschitz domains in
$\IR^3$ such that $\Upsilon$ is convex and $\Omega\subset\Upsilon$.
Let $\bT$ be a continuous, bijective and bi-Lipschitz
transformation\textemdash so that $\bT$ and
$\bT^{-1}$ belong to $\bW^{1,\infty}(\Upsilon)$\textemdash
mapping $\Upsilon$ onto itself. Then, it holds that
\begin{gather*}
\norm{U\circ\bT-U}{\Lp{\infty}{\Omega}}\leq \norm{\bT-\Id}{\lp{\infty}{\Upsilon}}\norm{U}{W^{1,\infty}(\Upsilon)},
\end{gather*}
for all $U\in W^{1,\infty}(\Upsilon)$.
\end{lemma}

\begin{lemma}\label{lem:trans_approx}
Let $\Upsilon$ and $\Omega$ be open Lipschitz domains in
$\IR^3$ such that $\Upsilon$ is convex and $\Omega\subset\Upsilon$.
Let $\bT$ be a continuous, bijective and bi-Lipschitz
transformation\textemdash $\bT$ and
$\bT^{-1}$ belong to $\bW^{1,\infty}(\Upsilon)$\textemdash
mapping $\Upsilon$ onto itself and such that
\begin{align}
{\sup_{\substack{\bx,\by\in\Upsilon \\ \bx\neq\by}}\frac{\norm{(\bT(\bx)-\bx)-(\bT(\by)-\by)}{\IR^3}}{\norm{\bx-\by}{\IR^3}}\leq\kappa<1\quad\mbox{and}\quad \vartheta^{-1}\leq\norm{\det{\rd\bT}}{\lp{\infty}{\Upsilon}}\leq\vartheta},\label{eq:t_cond_lem}
\end{align}
for some $\kappa\in (0,1)$ and $\vartheta>1$. Then, one has
\begin{gather*}
\norm{\bU\circ\bT-\bU}{0,{\Omega}}\leq (\vartheta^\half+1)\norm{\bT-\Id}{\Lp{\infty}{\Upsilon}}^s\norm{\bU}{s,{\Upsilon}},
\end{gather*}
for all $\bU\in\Hsob{s}{\Upsilon}$, with $0\leq s\leq 1$.
\end{lemma}

\subsection{Convergence of solution pull-backs in approximate domains}
\label{ssec:pullback_conv_smooth}
We begin by estimating the convergence to zero of the following approximation error:
\begin{align*}
\norm{\Psi_i\bE-\widetilde\bE_i}{\hcurl{\widetilde\D_i}},
\end{align*}
through an application of Strang's lemma \cite[Thm.~4.2.11]{SauterSchwabBEM}.
As in \cite{AJZS18,jerez2017electromagnetic}, we note that if
$\bE\in\hocurl{\D}$ is the unique solution of \cref{prob:varprob},
then $\Psi\bE\in\hocurl{\widetilde\D}$ is the unique solution of a modified Maxwell problem on $\widetilde\D$
arising from transferring the sesquilinear and antilinear forms $\Phi(\cdot,\cdot)$ and $\bF(\cdot)$ from
$\D$ to $\widetilde\D$ by a change of variables (\emph{cf.}~in \cite[Sec.~2.5.2]{AJZS18}). Specifically, for $\widetilde\D\in\kD$, we introduce the modified sesqulinear and antilinear forms as
\begin{align}\label{eq:pert_ses_ant}
\widehat\Phi(\bU,\bV):=\Phi(\Psi^{-1}\bU,\Psi^{-1}\bV)\quad\mbox{and}\quad \widehat\bF(\bV):=\bF(\Psi^{-1}\bV),
\end{align}
for all $\bU,\bV\in\hocurl{\widetilde\D}$.

\begin{problem}(Modified variational problem on $\widetilde\D$)
\label{prob:varprob_modified}
Find $\widehat\bE\in\hocurl{\widetilde\D}$ such that
\begin{align*}
\widehat\Phi(\widehat\bE,\bV)=\widehat\bF(\bV)\quad\forall\,\bV\in\hocurl{\widetilde\D}.
\end{align*}
\end{problem}

\begin{proposition}\label{prop:mod_prob_equiv}
Let \cref{ass:sesqform,ass:app_dom} hold and
let $\bE$ denote the solution of
\cref{prob:varprob}. Then, $\Psi\bE\in\hocurl{\widetilde\D}$ is the unique solution of \cref{prob:varprob_modified}.
\end{proposition}
\begin{proof}
Take $\bU,\bV\in\hocurl{\widetilde\D}$. Then, with $C_1>0$ as in \cref{ass:sesqform} and $C>0$ as in \cref{lem:curl_pull_cont}, we have that
\begin{align*}
\vert{\widehat\Phi(\bU,\bV)}\vert
&=\modulo{\Phi(\Psi^{-1}\bU,\Psi^{-1}\bV)}\leq C_1\norm{\Psi^{-1}\bU}{\hcurl{\D}}\norm{\Psi^{-1}\bV}{\hcurl{\D}}\\
&\leq C_1C^2\norm{\bU}{\hcurl{\widetilde\D}}	\norm{\bV}{\hcurl{\widetilde\D}},
\end{align*}
and
\begin{align*}
\vert{\widehat\bF(\bV)}\vert
&=\modulo{\bF(\Psi^{-1}\bV)}\leq\norm{\bF}{\hocurl{\D}'}\norm{\Psi^{-1}\bV}{\hcurl{\D}}
\leq C\norm{\bF}{\hocurl{\D}'}\norm{\bV}{\hcurl{\widetilde\D}}.
\end{align*}
Moreover, since $\Psi:\hocurl{\D}\to\hocurl{\widetilde\D}$ is an isomorphism, for every
$\bU\in\hocurl{\widetilde\D}$ it holds that
\begin{align*}
&\sup_{\bV\in\hocurl{\widetilde\D}\setminus\{\bnul\}}\frac{\vert{\widehat\Phi(\bU,\bV)}\vert}{\norm{\bU}{\hocurl{\widetilde\D}}\norm{\bV}{\hocurl{\widetilde\D}}}\\
&\geq C^{-2}
\sup_{\bV\in\hocurl{\widetilde\D}\setminus\{\bnul\}}\frac{\vert{\Phi(\Psi^{-1}\bU,\Psi^{-1}\bV)}\vert}{\norm{\Psi^{-1}\bU}{\hocurl{\widetilde\D}}\norm{\Psi^{-1}\bV}{\hocurl{\widetilde\D}}}\geq C^{-2}C_2,
\end{align*}
where the positive constant ${C}_2$ is as in \cref{ass:sesqform} and $C>0$ comes from \cref{lem:curl_pull_cont}.
Moreover, for every $\bV\in\hocurl{\widetilde\D}\setminus\{\bnul\}$ we have that
\begin{align*}
\sup\limits_{\bU\in\hocurl{\widetilde\D}\setminus\{\bnul\}}\vert{\widehat\Phi(\bU,\bV)}\vert
= \sup\limits_{\bU\in\hocurl{\D}\setminus\{\bnul\}}\vert{\Phi(\bU,\Psi^{-1}\bV)}\vert >0.
\end{align*}
Hence, since $\widehat\Phi(\cdot,\cdot)$ satisfies the inf-sup conditions,
we can conclude that \cref{prob:varprob_modified} is well posed and has
a unique solution in $\hocurl{\widetilde\D}$ \cite[Sec.~2.1.6]{SauterSchwabBEM}.
Moreover, since $\bE\in\hocurl{\D}$ solves \cref{prob:varprob} there
holds that
\begin{align*}
\widehat\Phi(\Psi\widetilde\bE,\bV)=\Phi(\widetilde\bE,\Psi^{-1}\bV)=\bF(\Psi^{-1}\bV)=\widehat\bF(\bV),
\end{align*}
for all $\bV\in\hocurl{\widetilde\D}$, and so $\Psi\bE\in\hocurl{\widetilde\D}$ is the unique solution of \cref{prob:varprob_modified}.
\end{proof}

\begin{theorem}\label{thm:pull_result_cont}
Let \cref{ass:sesqform,ass:extension,ass:sesqform_approx,ass:app_dom} hold
and let $\bE$ and $\bE_i$ denote the unique solutions of \cref{prob:varprob,prob:varprob_approx} on $\widetilde\D_i$ for each $i\in\IN$. Moreover, assume that
$\mu^{-1}$, $\epsilon$ and $\bJ$ have coefficients in $W^{1,\infty}(\D_H)$.
Then, it holds that
\begin{align*}
\norm{\Psi_i\bE-\widetilde\bE_i}{\hcurl{\widetilde\D_i}}\leq C d_1(\D_H,\bT_i)(\norm{\bE}{\hcurl{\D}}+\norm{\bJ}{\bm{W}^{1,\infty}(\D_H)}),
\end{align*}
where $C$ depends on $\omega$, $\mu$, $\epsilon$ and $\bJ$ but is independent of $i\in\IN$.
\end{theorem}
\begin{proof}
Fix $i\in\IN$, recall the sesquilinear and antilinear forms in
\cref{eq:pert_ses_ant} and let $\bU$, $\bV\in\hocurl{\widetilde\D_i}$.
We begin by noticing that the sesquilinear and antilinear
forms in \eqref{eq:pert_ses_ant} may be written as
\begin{align*}
\widehat\Phi_i(\bU,\bV)=\int_{\widetilde\D_i} \mu_{\bT_i}^{-1} \curl \mathbf{U} \cdot \curl \overline{\mathbf{V}}-
\omega^2\epsilon_{\bT_i} \mathbf{U} \cdot \overline{\mathbf{V}} \d\!\bx\quad\mbox{and}\quad\widehat\bF_i(\bV)=-\imath\omega\int_{\widetilde\D_i}\bJ_{\bT_i}\cdot\overline{\bV}\ddx,
\end{align*}
where
\begin{gather*}
\mu_{\bT_i}:=\det{\d\!\bT_i}\d\!\bT_i^{-1}(\mu\circ\bT_i)\d\!\bT_i^{-\top},\quad
\epsilon_{\bT_i}:=\det{\d\!\bT_i}\d\!\bT_i^{-1}(\epsilon\circ\bT_i)\d\!\bT_i^{-\top},\\
\bJ_{\bT_i}:=\det{\d\!\bT_i}\d\!\bT_i^{-1}(\bJ\circ\bT_i).
\end{gather*}
Then, one has that
\begin{align}
\begin{aligned}
&\modulo{\Phi(\bU,\bV)-\widehat\Phi_i(\bU,\bV)}\leq
\modulo{\int_{\widetilde\D_i} (\mu^{-1}-\mu_{\bT_i}^{-1}) \curl \mathbf{U} \cdot \curl \overline{\mathbf{V}} \d\!\bx}+\omega^2\modulo{\int_{\widetilde\D_i}
(\epsilon-\epsilon_{\bT_i}) \mathbf{U} \cdot \overline{\mathbf{V}} \d\!\bx}\\
&\leq C\left(\norm{\mu^{-1}-\mu_{\bT}^{-1}}{\Lp{\infty}{\D_H;\IC^{3\times 3}}}+\omega^2\norm{\epsilon-\epsilon_{\bT}}{\Lp{\infty}{\D_H;\IC^{3\times 3}}}\right)\norm{\bU}{\hcurl{\widetilde\D}}\norm{\bV}{\hcurl{\widetilde\D}}\\
&\leq C d_1(\D_H,\bT_i)(\norm{\epsilon}{\bm{W}^{1,\infty}(\D_H;\IC^{3\times 3})}+\norm{\mu^{-1}}{\bm{W}^{1,\infty}(\D;\IC^{3\times 3})})\norm{\bU}{\hcurl{\widetilde\D_i}}\norm{\bV}{\hcurl{\widetilde\D_i}},
\end{aligned}\label{eq:phihat_approx_phi}
\end{align}
where $C>0$ in the first inequality follows from the norm equivalence on finite-dimensional spaces and is independent of $i\in\IN$. The last bound is derived by applying \cref{lem:error_infty} and \cref{ass:app_dom}, whereby the positive constant $C$ depends on $\vartheta>1$ in \cref{ass:app_dom}. 

Analogously, we have that
\begin{align*}
\vert{\bF(\bV)-\widehat\bF_i(\bV)}\vert
\leq C d_1(\D_H,\bT_i)\norm{\bJ}{\bm{W}^{1,\infty}(\D_H)}\norm{\bV}{\hcurl{\widetilde\D_i}},
\end{align*}
where $C>0$ is independent of $i\in\IN$.
Finally, by \cref{ass:sesqform_approx} and \cref{prop:mod_prob_equiv}, it holds that
\begin{align*}
&\widetilde C_s\norm{\Psi_i\bE-\widetilde\bE_i}{\hcurl{\widetilde\D_i}}\leq 
\sup_{\bV\in \hocurl{\widetilde\D_i}\setminus\{\bnul\}} 
\frac{\vert{\Phi(\Psi_i\bE-\widetilde\bE_i,\bV)}\vert}{\norm{\bV}{\hcurl{\widetilde\D_i}}}\\
&\leq \sup_{\bV\in \hocurl{\D}\setminus\{\bnul\}} 
\frac{\vert{\Phi(\Psi_i\bE,\bV)-\widehat\Phi_i(\Psi_i\bE,\bV)}\vert+\vert{\widehat\bF_i(\bV)-\bF(\bV)}\vert}{\norm{\bV}{\hcurl{\widetilde\D_i}}}.\\
&\leq Cd_1(\D_H;\bT_i)\!\left[(\omega^2\norm{\epsilon}{\bm{W}^{1,\infty}(\D_H;\IC^{3\times 3})}+\norm{\mu^{-1}}{\bm{W}^{1,\infty}(\D;\IC^{3\times 3})})\norm{\Psi_i\bE}{\hcurl{\widetilde\D_i}}+\omega\norm{\bJ}{\bm{W}^{1,\infty}(\D_H)})\right],
\end{align*}
and the claimed estimate then follows by \cref{lem:curl_pull_cont}.
\end{proof}

\subsection{Convergence to an extended solution over \texorpdfstring{$\D_H$}{Lg}}\label{ssec:extended_conv}
Throughout this section, we aim at deriving convergence rate estimates for the approximation error
\begin{align*}
\norm{\bE-\widetilde\bE_i}{\hcurl{\widetilde{\D}_i}}
\end{align*}
To this end, we will require $\bE$ to be extended to $\D_H$ with some higher regularity $\hscurl{\D_H}{r}$, for some $r\in (0,1]$---see \cref{ass:first_approach} below.
Our results in \cref{ssec:pullback_conv_smooth} require no such assumption.

\begin{assumption}[Extension to $\D_H$]\label{ass:first_approach}
$\bE\in\hocurl{\D}$\textemdash the solution of
\cref{prob:varprob}\textemdash may be extended to $\hscurl{\D_H}{r}$, with $r\in (0,1]$.
We slightly abuse notation by referring to the extension of $\bE$ as $\bE$ as well. Furthermore,
we assume the hold-all domain $\D_H$ to be convex.
\end{assumption}

For an example of an extension operator from $\hscurl{\D}{1}$ to
$\hscurl{\D_H}{1}$ under the condition that $\D$ be a domain of class
$\C^2$ (\emph{cf.~}\cite[Thm.~2]{Hiptmair_2012}).

\begin{theorem}\label{thm:first_result_cont}
Let \cref{ass:sesqform,ass:extension,ass:sesqform_approx,ass:app_dom,ass:first_approach} hold
and let $\bE$ and $\widetilde\bE_i$ denote respectively the unique solutions of \cref{prob:varprob,prob:varprob_approx} on $\widetilde\D_i$ for all $i\in\IN$. Then, it holds that
\begin{align}\label{eq:class_approach_final_estimate}
\norm{\bE-\widetilde\bE_i}{\hcurl{\widetilde\D_i}}\leq
C\left[(d_1(\D_H,\bT_i)+d_0(\D_H,\bT_i)^{r})\norm{\bE}{\hscurl{\D_H}{r}}+d_1(\D_H,\bT_i)\norm{\bJ}{\bm{W}^{1,\infty}(\D_H)}\right],
\end{align}
for all $i\in\IN$, where the positive constant $C$ depends on $\omega$, $\mu$, $\epsilon$ and $\bJ$ but is independent of $i\in\IN$.
\end{theorem}
\begin{proof}
Fix $i\in\IN$ and recall $\Psi_i:\hocurl{\D}\to\hocurl{\widetilde\D_i}$ as in \cref{lem:trans_curl}. Then, by the triangle inequality, one has
\begin{align}
\norm{\bE-\widetilde\bE_i}{\hcurl{\widetilde\D_i}}\leq \norm{\bE-\Psi_i\bE}{\hcurl{\widetilde\D_i}}+\norm{\Psi_i\bE-\widetilde\bE_i}{\hcurl{\widetilde\D_i}},\label{eq:Thm_firstapp_1}
\end{align}
where $\bE$ is chosen to be extended by \cref{ass:first_approach}.
We start by bounding the first term in the right-hand side of \cref{eq:Thm_firstapp_1} thanks to \cref{lem:trans_approx}:
\begin{align*}
\norm{\bE-\Psi_i\bE}{0,{\widetilde\D_i}}&=\norm{\bE-\rd\bT_i^{\top}\bE\circ\bT_i}{0,{\widetilde{\D}_i}}\nonumber\\
&\leq \norm{\rd\bT_i^\top-\Id}{\Lp{\infty}{\D_H}}\norm{\bE}{0,{\D_H}}+(\vartheta^\half+1)\norm{\rd\bT_i^{\top}}{\Lp{\infty}{\D_H}}\norm{\bT_i-\Id}{\Lp{\infty}{\D_H}}^{r}\norm{\bE}{r,{\D_H}}\nonumber\\
&\leq \left(\norm{\rd\bT_i^\top-\Id}{\Lp{\infty}{\D_H}}+(\vartheta^\half+1)\vartheta\norm{\bT_i-\Id}{\Lp{\infty}{\D_H}}^{r}\right)\norm{\bE}{{r},{\D_H}}.
\end{align*}
Simarly, we have that
\begin{align*}
\norm{\curl(\bE\!-\!\Psi_i\bE)}{0,{\widetilde\D_i}}\!\leq\! \left(\norm{\rd\bT_i^\co\!-\!\Id}{\Lp{\infty}{\D_H}}\!+\!(\vartheta^\half\!+\!1)\vartheta\norm{\bT_i\!-\!\Id}{\Lp{\infty}{\D_H}}^{r}\right)\norm{\curl\bE}{{r},{\D_H}},
\end{align*}
so that, by \cref{ass:app_dom}, we have that
\begin{align*}
\norm{\bE-\Psi_i(\bE)}{\hcurl{\widetilde\D_i}}^2&=\norm{\bE-\Psi_i(\bE)}{0,{\widetilde\D_i}}^2+\norm{\curl\bE-\curl\Psi_i(\bE)}{0,{\widetilde\D_i}}^2\\
&\leq  C^2\left(d_1(\D_H,\bT_i)+(\vartheta^\half+1)\vartheta d_0(\D_H,\bT_i)^{r}\right)^2\left(\norm{\bE}{{r},{\D_H}}^2+\norm{\curl\bE}{{r},{\D_H}}^2\right)\\
&\leq  C^2(\vartheta^\half+1)^2\vartheta^2\left(d_1(\D_H,\bT_i)+d_0(\D_H,\bT_i)^{r}\right)^2\norm{\bE}{\hscurl{\D_H}{r}}^2,
\end{align*}
where $C$ arises from \cref{eq:CboundMat} and is independent of $i\in\IN$. Furthermore, the second term in the right-hand side of \cref{eq:Thm_firstapp_1} may be bounded by direct application of \cref{thm:pull_result_cont}, yielding
\begin{align*}
\norm{\bE-\widetilde\bE_i}{\hcurl{\widetilde\D_i}}\leq &\ C\left(d_1(\D_H,\bT_i)+d_0(\D_H,\bT_i)^{r}\right)\\
&\cdot \norm{\bE}{\hscurl{\D_H}{r}}+Cd_1(\D_H,\bT_i)(\norm{\bE}{\hcurl{\D}}+\norm{\bJ}{\bm{W}^{1,\infty}(\D_H)}).
\end{align*}
The stated result follows by reordering the terms in the last equation.
\end{proof}

\section{Variational problems on approximate domains: discrete problem}
\label{sec:DiscProb}
We now analyze the discrete version of \cref{prob:varprob_approx},
and consider the family of approximate domains $\kD$ corresponding to curved
domains $\{\D_{h_i}\}_{i\in\IN}$ introduced in \cref{sec:FE}, i.e., $\D_{h_i}\equiv\widetilde\D_i$. This discrete setting is signaled by denoting an
arbitrary element of $\kD$ by $\D_h$ instead of $\widetilde\D$. We also recall the discrete
space $\bm{P}^c_0(\tau_h)$ in \cref{eq:disc_space} as the space
of curl-conforming piecewise polynomials of degree $k\in\IN$ 
with null flipped-Dirichlet trace on $\Gamma$.

\begin{assumption}[Discrete inf-sup conditions]
\label{ass:sesqform_approx_disc}
Assume the sesquilinear form in 
\cref{eq:Phi_approx} to satisfy the following conditions:
\begin{align}
&\inf_{\bU\in \bm{P}^c_{0}(\tau_h)\setminus\{\bnul\}}\left(\sup_{\bV\in \bm{P}^c_{0}(\tau_h)\setminus\{\bnul\}} 
\frac{\modulo{\Phi(\bU,\bV)}}{\norm{\bU}{\hcurl{{\D_h}}}\norm{\bV}{\hcurl{{\D_h}}}}\right)\geq C_2,\nonumber\\
&\inf_{\bV\in \bm{P}^c_{0}(\tau_h)\setminus\{\bnul\}}\left(\sup_{\bU\in \bm{P}^c_{0}(\tau_h)\setminus\{\bnul\}} 
\frac{\modulo{\Phi(\bU,\bV)}}{\norm{\bU}{\hcurl{{\D_h}}}\norm{\bV}{\hcurl{{\D_h}}}}\right)\geq C_3,\label{eq:inf_sup_alternative}
\end{align}
for all $\tau_h\in\kT$ ($\D_h\in\kD$), where the positive constants
${C}_{2}$ and ${C}_{3}$ are independent of the mesh-size $h$.
\end{assumption}

\begin{problem}[Discrete variational problem on inexact domains]
\label{prob:varprob_approx_disc}%
Find $\bE_h\in \bm{P}^c_0(\tau_h)$ such that
\begin{align*}
\Phi(\bE_h,\bV)=\bF(\bV),
\end{align*}
for all $\bV\in \bm{P}^c_0(\tau_h)$.
\end{problem}

\cref{ass:sesqform_approx_disc} ensures uniqueness and existence
of solutions of \cref{prob:varprob_approx_disc}, whose solutions are denoted $\bE_h\in\bm{P}^c_0(\tau_h)$ for general $\tau_h\in\kT$ and $\bE_{h_i}\in\bm{P}^c_0(\tau_{h_i})$ for each $i\in\IN$, respectively. Note that the condition in \cref{eq:inf_sup_alternative} is stronger than required for the purposes of proving the unique solvability of \cref{prob:varprob_approx_disc}. This stronger condition is necessary to prove a discrete analogue of \cref{prop:mod_prob_equiv} (see \cref{prop:mod_disc_prob}) via a perturbation argument.

\subsection{Convergence of domains in a discrete setting}
Let us start with the following result, regarding the approximation of
functions in $\hocurl{\D}$ by discrete functions in $\bm{P}^c_0(\tau_{h})$.
Here, once again we identify elements of $\hocurl{\D}$ and $\hocurl{\D_h}$
with their zero-extensions to $\D_H$.

\begin{lemma}
\label{lem:HausImpDiscMosc}
Suppose $\kD$ approximates $\D$ in the sense of Hausdorff and that
$\D$ and all $\D_h\in\kD$ are Lipschitz continuous domains
with uniform Lipschitz constant in $\kD$. Then, it holds that
\begin{enumerate}[label=(\alph*)]
\item For every $\bU\in\hocurl{\D}$ there exists a sequence $\{\bU_{i}\}_{i\in\IN}$ with
$\bU_i\in\bm{P}^c_0(\tau_{h_i})$ for all $i\in\IN$, such that $\bU_i$ converges to $\bU$
strongly in $\hocurl{\D_H}$.\label{it:1def:MoscoDisc}
\item Weak limits of every sequence $\{\bU_i\}_{i\in\IN}$, with $\bU_i\in\bm{P}^c_0(\tau_{h_i})$, belong to $\hocurl{\D}$.\label{it:2def:MoscoDisc}
\end{enumerate}
\end{lemma}
\begin{proof}
\cref{it:1def:MoscoDisc}: Take $\bU\in\hocurl{\D}$
and set an arbitrary $\epsilon>0$. By density of $\bm\C^\infty_0(\D)$ in
$\hocurl{\D}$, there exists some $\widetilde\bU_\epsilon\in\bm\C^\infty_0(\D)$
such that $\norm{\widetilde\bU_\epsilon-\bU}{\hocurl{\D_H}}\leq \frac{\epsilon}{2}$. Moreover, there exists some $i_\epsilon\in\IN$ such that $\supp(\widetilde\bU_\epsilon)\subset\D_{h_i}$ for all $i>i_\epsilon$ (\emph{cf.~}proof of Lemma 3 in \cite[Sec.~3]{Pironneau:1984}), and another
$i'_{\epsilon}\in\IN$ such that $\norm{\mathbf{\Pi}_{h_i}\widetilde\bU_\epsilon-\widetilde\bU_\epsilon}{\hcurl{\D_H}}\leq\frac{\epsilon}{2}$ for all $i>i'_{\epsilon}$. Then, it holds that 
\begin{align*}
\norm{\mathbf{\Pi}_{h_i}\widetilde\bU_\epsilon-\bU}{\hcurl{\D_H}}\leq\epsilon,
\end{align*}
for all $i>i'_{\epsilon}$. Repeating the procedure above for a decreasing
sequence of $\epsilon>0$ allows the construction of a strongly convergent sequence to $\bU\in\hocurl{\D}$ (\emph{cf.}~\cite[Lem.~2.4]{chandler2021boundary}).

\cref{it:2def:MoscoDisc}: Follows from \cref{lem:HausImpMosc} by
the inclusion $\bm{P}^c_0(\tau_{h_i})\subset\hocurl{\D_{h_i}}$.
\end{proof}

In order to obtain convergence rates of finite element
solutions to $\bE$---thereby proving discrete analogues to \cref{thm:pull_result_cont,thm:first_result_cont}---we need one further assumption on the transformations $\{\bT_i\}_{i\in\IN}$.

\begin{assumption}
\label{ass:app_dom_T}
For each element of the family $\{\bT_i\}_{i\in\IN}$ mapping $\D_{h_i}$ to $\D$, as given in
\cref{ass:app_dom}, we assume that the transformations and their inverses belong to
$\bm{W}^{\K+1,\infty}(K)$ for each tetrahedron $K$ in their
corresponding mesh, i.e., for each $i\in\IN$, we assume that
\begin{align*}
\bT_i\vert_{K}\in\bm{W}^{\K+1,\infty}(K)
\end{align*}
for all $K\in\tau_{h_i}$, and that
\begin{gather*}
\lim\limits_{i\rightarrow\infty}\max\limits_{K\in\tau_{h_i}}\norm{\bT_i\vert_{K}-\Id}{\bm{W}^{\K,\infty}(K)}+\max\limits_{K\in\tau_{h_i}}\norm{\bT_i^{-1}\vert_{K}-\Id}{\bm{W}^{\K,\infty}(K)}=0,\\
\max\limits_{K\in\tau_{h_i}}\norm{\bT_i\vert_{K}}{\bm{W}^{\K+1,\infty}(K)}+\max\limits_{K\in\tau_{h_i}}\norm{\bT_i^{-1}\vert_{K}}{\bm{W}^{\K+1,\infty}(K)}\leq c_\kT,
\end{gather*}
for all $i\in\IN$, where $c_\kT$ is a positive constant
independent of $i\in\IN$. Moreover, the family $\{\bT_i\}_{i\in\IN}$
satisfies the following bound:
\begin{align}\label{eq:Tboundh}
d_0(\D_H,\bT_i)\leq C h_i^{\K+1}\quad\mbox{and}\quad d_1(\D_H,\bT_i)\leq C h_i^{\K},
\end{align}
for $\K\in\set{1}{\mathfrak{M}-1}$ as in \cref{ass:mesh}, with $C>0$ independent of $i\in\IN$.
\end{assumption}

\cref{ass:app_dom_T} is justified by constructions of
mappings $\{\bT_i\}_{i\in\IN}$ in the context of curved meshes in finite
elements (\emph{cf.}~\cite[Prop.~2 and 3]{lenoir1986} for
properties of these mappings and Sections 3 and 5 therein for their construction
satisfying \cref{ass:mesh,ass:app_dom_T}). Also, \cref{ass:dom} related to smoothness requirements on $\D$ to be of class $\C^\mathfrak{M}$,
has no direct relevance on the coming proofs of our discrete approximation results. Indeed,
the smoothness of the domain only plays a role in proving decay rates such as \cref{eq:Tboundh}
in practical constructions of the transformations $\{\bT_i\}_{i\in\IN}$, which we have assumed through \cref{ass:app_dom_T}.
Nonetheless, we opt to enforce \cref{ass:dom} in our coming results to emphasize its necessity for the construction of the transformations $\{\bT_i\}_{i\in\IN}$ and, therefore, our main results.

\subsection{Discrete convergence of solution pull-backs in approximate domains}
We now extend the results in \cref{ssec:pullback_conv_smooth} to our discrete setting, i.e.~we estimate the error
\begin{align*}
\norm{\Psi_i\bE-\bE_{h_i}}{\hcurl{\D_{h_i}}},
\end{align*}
as $i$ grows towards infinity, where $\Psi_i:\hocurl{\D}\to\hocurl{\D_{h_i}}$ was introduced in \cref{lem:trans_curl}.

\begin{problem}[Modified discrete variational problem on $\D_h$]
\label{prob:varprob_modified_disc}
Find $\widehat\bE_{h}\in {\bm P}^c_0(\tau_{h})$ such that
\begin{align*}
\widehat\Phi(\widehat\bE_h,\bV)=\widehat\bF(\bV)\quad\forall\,\bV\in {\bm P}^c_0(\tau_{h}).
\end{align*}
\end{problem}

\begin{proposition}\label{prop:mod_disc_prob}
Let \cref{ass:sesqform,ass:mesh,ass:app_dom,ass:sesqform_approx_disc} hold.
Moreover, assume that
$\mu^{-1}$, $\epsilon$ and $\bJ$ have coefficients in $W^{1,\infty}(\D_H)$.
Then, \cref{prob:varprob_modified_disc} is well posed for all sufficiently small $h>0$.
\end{proposition}
\begin{proof}
The continuity of $\widehat\Phi(\cdot,\cdot)$ and $\widehat\bF(\cdot)$ on $\bm{P}^c_0(\tau_h)$ follows from the
proof of \cref{prop:mod_prob_equiv} and the conformity of the finite element space, i.e.~$\bm{P}^c_0(\tau_h)\subset\hocurl{\D_h}$. The inf-sup conditions on $\widehat\Phi(\cdot,\cdot)$ follow by a perturbation argument through
\cref{eq:phihat_approx_phi} and our assumptions\footnote{See, for example, the proof of \cite[Thm.~{4.2.11}]{SauterSchwabBEM}}.
\end{proof}

In order to estimate the convergence rates for approximations of $\bE_{h_i}$ to $\Psi_i\bE$ as $i\in\IN$ grows
to infinity, one needs to show that $\Psi_i\bE$ preserves the smoothness of $\bE$ to some degree.

\begin{lemma}\label{lem:smooth_pull}
Let \cref{ass:mesh,ass:app_dom_T} hold.
Let $\bU\in\hocurl{\D}\cap\hscurl{\D}{s}$ for some $s\in\{1:\K\}$ and let $K\in\tau_{h}$ be an arbitrary tetrahedron
of the mesh $\tau_{h}\in\kT$. Then, for $\Psi:\hocurl{\D}\to\hocurl{\D_h}$ as introduced in \cref{lem:trans_curl},
there holds that $\Psi\bU\in\hscurl{K}{s}$ and that 
\begin{align*}
\norm{\Psi_i(\bU)}{\hscurl{K}{s}}\leq \vartheta^\half C\norm{\bU}{\hscurl{\bT_{i}(K)}{s}},
\end{align*}
where the positive constant $C$ depends on $c_\kT$ in
Assumption \ref{ass:app_dom_T} but not on the mesh-size or $\bU\in\hocurl{\D}\cap\hscurl{\D}{s}$.
\end{lemma}
\begin{proof}
The result is a direct consequence of \cite[Lem.~1]{ciarlet1972combined}
or \cite[Lem.~3]{ciarlet1972interpolation} together with our assumptions.
\end{proof}

\begin{theorem}\label{thm:pull_result_disc}
Let \cref{ass:dom,ass:sesqform,ass:mesh,ass:extension,ass:sesqform_approx,ass:app_dom,ass:sesqform_approx_disc,ass:app_dom_T} hold,
let $\bE$ and $\bE_{h_i}$ denote the unique solutions of \cref{prob:varprob,prob:varprob_approx_disc}, respectively, and assume that $\bE\in\hscurl{\D}{s}$ for some $s\in\{1:k\}$.
Furthermore, let $\mu^{-1}$, $\epsilon$ and $\bJ$ have coefficients in $W^{1,\infty}(\D)$.
Then, there exists some $\mathfrak{i}\in\IN$ such that, for all $i>\mathfrak{i}$, it holds that
\begin{align}\label{eq:pull_approach_disc_estimate}
\norm{\Psi_i\bE-\bE_{h_i}}{\hcurl{\D}}\leq C (h_i^s\norm{\bE}{\hscurl{\D}{s}}+h_i^\K\norm{\bJ}{\bm{W}^{1,\infty}(\D_H)}),
\end{align}
where $C>0$ depends on $\omega$, $\mu$, $\epsilon$ and $\bJ$ but is independent of $i\in\IN$.
\end{theorem}
\begin{proof}
Fix $i\in\IN$ large enough so that the results of \cref{prop:mod_disc_prob} hold true and let $\widehat\bE_i\in\bm{P}^c_0(\tau_{h_i})$ denote the solution of
\cref{prob:varprob_modified_disc}. Then, one can write
\begin{align*}
\norm{\Psi_i\bE-\bE_{h_i}}{\hcurl{\D_{h_i}}}\leq \norm{\Psi_i\bE-\widehat\bE_{h_i}}{\hcurl{\D_{h_i}}}+\norm{\widehat\bE_{h_i}-\bE_{h_i}}{\hcurl{\D_{h_i}}}.
\end{align*}
\Cref{lem:smooth_pull} implies that $\Psi_i\bE\in\hscurl{K}{s}$ for all $K\in\tau_{h_i}$, with
\begin{align*}
\norm{\Psi_i(\bE)}{\hscurl{K}{s}}\leq C\norm{\bE}{\hscurl{\bT_{i}(K)}{s}},
\end{align*}
for $C$ positive independent of $i\in\IN$.
Since $\Psi_i\bE\in\hocurl{\D_{h_i}}$ is the solution of \cref{prob:varprob_modified} (see \cref{prop:mod_prob_equiv}), Theorem 5.41 in \cite{Monk:2003aa} yields
\begin{align*}
\norm{\Psi_i\bE-\widehat\bE_{h_i}}{\hcurl{\D_{h_i}}} &\leq C{h_i}^s\left(\sum\limits_{K\in\tau_{h_i}}\norm{\Psi_i\bE}{\hscurl{K}{s}}^2\right)^\half \\
&\leq C{h_i}^s\left(\sum\limits_{K\in\tau_{h_i}}\norm{\bE}{\hscurl{\bT_i(K)}{s}}^2\right)^\half\leq C{h_i}^s\norm{\bE}{\hscurl{\D}{s}},
\end{align*}
where the positive constant $C$ is not necessarily equal in each appearance.
Furthermore, by a reasoning analogous to that in \cref{thm:pull_result_cont}, followed by an
application of \cref{lem:curl_pull_cont} together with \cref{ass:app_dom_T} and the wellposedness of \cref{prob:varprob_modified_disc} (also, see \cref{prop:mod_prob_equiv}), we have that
\begin{align*}
\norm{\widehat\bE_{h_i}-\bE_{h_i}}{\hcurl{\D_{h_i}}}&\leq C d_1(\D_H,\bT_i)(\norm{\widehat\bE_{h_i}}{\hcurl{\D_{h_i}}}+\norm{\bJ}{\bm{W}^{1,\infty}(\D_H)})
\leq C h_i^{\K}\norm{\bJ}{\bm{W}^{1,\infty}(\D_H)}
\end{align*}
where again $C$ may vary, but remains $h_i$-independent.
The result in \cref{eq:pull_approach_disc_estimate} is then deduced by
a simple combination of the above estimates.
\end{proof}

\subsection{Discrete convergence to an extended solution over \texorpdfstring{$\D_H$}{Lg}}
\label{ssec:extension_convergence_disc}
We now transfer our results in \cref{ssec:extended_conv} to our discrete setting.
\begin{theorem}\label{thm:first_result_disc}
Let \cref{ass:dom,ass:sesqform,ass:mesh,ass:extension,ass:sesqform_approx,ass:app_dom,ass:sesqform_approx_disc,ass:app_dom_T,ass:first_approach} hold, let $\bE$ and $\bE_{h_i}$ denote the unique solutions of
\cref{prob:varprob,prob:varprob_approx_disc} and assume that $\bE\in\hscurl{\D}{s}$ for some $s\in\{1:k\}$.
Furthermore, let $\mu^{-1}$, $\epsilon$ and $\bJ$ have coefficients in $W^{1,\infty}(\D)$. By \cref{ass:first_approach} $\bE\in\hscurl{\D_H}{r}$ for some $r\in(0,1]$. Then, there exists some $\mathfrak{i}\in\IN$ such that, for all $i>\mathfrak{i}$, it holds that
\begin{align*}
\begin{aligned}
\norm{\bE-\bE_{h_i}}{\hcurl{\D_{h_i}}}\leq
C\bigg{[}(h_i^{\K}\!+\!h_i^{r(\K+1)})\norm{\bE}{\hscurl{\D_H}{r}}
\!+\!h_i^s\norm{\bE}{\hscurl{\D}{s}}
\!+\!h_i^\K\norm{\bJ}{\bm{W}^{1,\infty}(\D_H)}
\bigg{]},
\end{aligned}
\end{align*}
where $C>0$ depends on $\omega$, $\mu$, $\epsilon$ and $\bJ$, but is independent of $i\in\IN$.
\end{theorem}
\begin{proof}
Fix $i\in\IN$ large enough so that the results of \cref{prop:mod_disc_prob} hold true. The triangle inequality yields
\begin{align*}
\norm{\bE-\bE_{h_i}}{\hcurl{\D_{h_i}}}
&\leq \norm{\bE-\Psi_i\bE}{\hcurl{\D_{h_i}}}+
\norm{\Psi_i\bE-\bE_{h_i}}{\hcurl{\D_{h_i}}}.
\end{align*}
From the proof of \cref{thm:first_result_cont}, it follows that
\begin{align*}
\norm{\bE-\Psi_i\bE}{\hcurl{\D_{h_i}}}\leq C\left(d_1(\D_H,\bT_i)+d_0(\D_H,\bT_i)^{r}\right)\norm{\bE}{\hscurl{\D_H}{r}},
\end{align*}
which, together with \cref{ass:app_dom_T}, gives
\begin{align*}
\norm{\bE-\Psi_i\bE}{\hcurl{\D_{h_i}}}\leq C\left(h_i^{\K}+h_i^{r(\K+1)}\right)\norm{\bE}{\hscurl{\D_H}{r}},
\end{align*}
where $C$ is a positive constant depending only on $\vartheta$ in \cref{ass:app_dom}. Moreover,
a direct application of \cref{thm:pull_result_disc} leads to
\begin{align*}
\norm{\Psi_i\bE-\bE_{h_i}}{\hcurl{\D_{h_i}}}
\leq C (h_i^s\norm{\bE}{\hscurl{\D}{s}}+h_i^\K\norm{\bJ}{\bm{W}^{1,\infty}(\D_H)}),
\end{align*}
where the positive constant $C$ follows from \cref{thm:pull_result_disc}.
The result then follows by a straightforward combination of previous estimates.
%
\end{proof}


\subsection{A fully discrete estimate}
We now deduce convergence estimates for the fully
discrete Maxwell variational problem under consideration
by incorporating our findings in \cite{aylwin2020effect}.
To this end, let us introduce quadrature rules for the numerical
approximation of \cref{prob:varprob_approx_disc}.
\begin{definition}[Quadratures]\label{def:quad}
For $L\in\IN$, let $\{\breve{w}_{l}\}_{l=1}^{L}\subset \IR$ be a set of
\emph{quadrature weights} and let $\{\breve{\bm{b}}_{l}\}_{l=1}^{L}\subset \rK$
be a set of corresponding \emph{quadrature points}. Then, we introduce the following linear operator
over $\C(\breve{K})$,
\begin{align*}
 Q_{\breve{K}}:\left\lbrace\begin{aligned}\C(\breve{K}) &\to\IC\\
 \phi &\mapsto \sum_{l=1}^{L}\breve{w}_{l}\phi(\breve{\bm{b}}_{l})\end{aligned}\right.
\end{align*}
Moreover, we say $Q_{\breve{K}}$ is exact on polynomials of degree $\mathfrak{n}\in\IN_0$ if and only if
\begin{align*}
 Q_{\breve{K}}(\phi)=\int_{\breve{K}}\phi(\bx)\ddx\quad\forall\;\phi\in\mathbb{P}_{\mathfrak{n}}(\breve{K};\IC).
\end{align*}
\end{definition}
Quadratures on arbitrary mesh elements $K\in\tau_h$ are defined from a quadrature
on $\breve{K}$ through the mappings in \cref{ass:mesh} as follows
\begin{align}\label{eq:Qcurv}
Q_{K}(\phi):=Q_{\breve{K}}(\det{\rd\bmT_K}\phi\circ\bmT_K).
\end{align}
Then, for three distinct quadrature rules over $\breve K$\textemdash denoted
as $Q_\rK^1$, $Q_\rK^2$ and $Q_\rK^3$ and which shall be specified later on\textemdash we introduce the numerical
approximations of the sesquilinear and antilinear forms in
\eqref{eq:Phi} as
\begin{align}
&{\Phi}_{h}({\bU_{h}},{\bV_{h}}):=\sum_{ K\in\tau_{h}}
Q_{K}^{1}(\mu^{-1}\curl\bU_{h}\cdot\curl\overline{\bV_{h}})
+Q_{K}^{2}(-\omega^2\epsilon\bU_{h}\cdot\overline{\bV_{h}}),\label{eq:numericPhi}\\
&{\bF}_{h}({\bV_{h}}):=\sum_{ K\in\tau_{h}}Q_{ K}^3(-\imath\omega\bJ\cdot \overline{\bV_{h}}),\label{eq:numericrhs}
\end{align}
for all $\bU_h$ and $\bV_h$ in $\bm{P}^c_0(\tau_h)$, where $Q_{K}^{i}$ is constructed
from $Q_{\breve K}^{i}$ through \eqref{eq:Qcurv} for all $i\in\{1:3\}$.

\begin{problem}[Fully discrete problem]
\label{prob:varprobnumcurv}
Find $\widetilde{\bE}_h\in\bm P^c_0({\tau}_h)$ such that
\begin{align*}
\Phi_{h}(\widetilde{\bE}_h,{\bV}_h)=
\bF_{h}({\bV}_h),
\end{align*}
for all ${\bV}_h\in\bm P^c_0(\tau_h)$.
\end{problem}

\subsubsection{Fully discrete convergence of solution pull-backs in approximate domains}

We now present a fully discrete version of \cref{thm:pull_result_disc}, stating the approximation properties
of the solution of \cref{prob:varprobnumcurv} to the pull-back of the solution of \cref{prob:varprob}. To limit the number of parameters with effects on the convergence rate, we restrict ourselves to the case of isoparametric finite elements ($\K=k$).

\begin{theorem}\label{thm:pull_result_disc_num}
Let all the assumptions  in \cref{thm:pull_result_disc} hold and take $\K=k$.
Additionally, let us assume that there holds that
\begin{align*}
\bJ\in\bW^{s,q}(\D_H)\quad\mbox{and}\quad \epsilon_{i,j},\;(\mu^{-1})_{i,j}\in W^{s,\infty}(\D_H)\quad\forall\;i,\;j\in\{1:3\},
\end{align*}
for some $q>\min(2,s/3)$, where $s\in\{1:k\}$ is as in \cref{thm:pull_result_disc}, as well as the following conditions on the quadrature rules defining $\Phi_h(\cdot, \cdot)$ and $\bF_h(\cdot)$ in \cref{eq:numericPhi,eq:numericrhs}, respectively,
\begin{itemize}
\item $Q^1_{\rK}$ is exact for polynomials
of degree $2k+s-3$ and
\item $Q^2_{\rK}$ and $Q^3_{\rK}$ are exact for polynomials
of degree $3k+s-3$.
\end{itemize}
Then, there exists some $\mathfrak{i}\in\IN$ such that for all
$i\in\IN$ with $i>\mathfrak{i}$, \cref{prob:varprobnumcurv}
is uniquely solvable and its solution, denoted $\widetilde\bE_{h_i}$, satisfies
\begin{align*}
\norm{\Psi_i\bE-\widetilde\bE_{h_i}}{\hcurl{\D_{h_i}}}\leq
Ch_i^s\left(\norm{\bE}{\hscurl{\D}{s}}+\norm{\bJ}{\bm{W}^{s,q}(\D_H)}+\norm{\bJ}{\bm{W}^{1,\infty}(\D_H)}\right),
\end{align*}
where the positive constant $C$ is independent of the mesh-size.
\end{theorem}
\begin{proof}
Fix $i\in\IN$ as in the proof of \cref{thm:pull_result_disc} and let $\bE_{h_i}\in\bm{P}^c_0(\tau_{h_i})$ denote the
solution of \cref{prob:varprob_approx_disc}.
Performing small modifications of \cite[Thm.~4]{aylwin2020effect} (\emph{cf.}~proof of \cite[Thm.~4.2.11]{SauterSchwabBEM}),
it follows that
\begin{align*}
\norm{\Psi_i\bE-\widetilde\bE_{h_i}}{\hcurl{\D_{h_i}}} 
\leq C\Bigg[
\norm{\Psi_i(\bE)-\bE_{h_i}}{\hcurl{\D_{h_i}}}+\norm{\mathbf{\Pi}_{h_i}(\Psi_i\bE)-\bE_{h_i}}{\hcurl{\D_{h_i}}}\\
+
h_i^s\left(\sum\limits_{K\in\tau_{h_i}}\norm{\mathbf{\Pi}_{h_i}(\Psi_i\bE)}{\hscurl{K}{s}}^2\right)^\half+h_i^s\norm{\bJ}{\bm{W}^{s,q}(\D_H)} \Bigg],
\end{align*}
where $C>0$ depends on the problem's parameters, including the constants in \cref{ass:sesqform_approx,ass:sesqform_approx_disc}, but it is independent of
$i\in\IN$. \cref{thm:pull_result_disc} then yields
\begin{align*}
\norm{\Psi_i(\bE)-\bE_{h_i}}{\hcurl{\D_{h_i}}}\leq C h_i^s(\norm{\bE}{\hscurl{\D}{s}}+\norm{\bJ}{\bm{W}^{1,\infty}(\D_H)}),
\end{align*}
with $C$ from \cref{thm:pull_result_disc} independently of $i\in\IN$.
The continuity of the global interpolation operator (\cref{prop:locinterp_cont,defi:glob_interp}) leads to
\begin{align*}
\norm{\mathbf{\Pi}_{h_i}(\Psi_i(\bE))-\bE_{h_i}}{\hcurl{\D_{h_i}}}=\norm{\mathbf{\Pi}_{h_i}(\Psi_i(\bE)-\bE_{h_i})}{\hcurl{\D_{h_i}}}\leq C\norm{\Psi_i(\bE)-\bE_{h_i}}{\hcurl{\D_{h_i}}},
\end{align*}
where the positive constant $C$ is as in \cref{prop:locinterp_cont} and is also independent of $i\in\IN$.
Moreover, by \cref{prop:locinterp_cont,lem:smooth_pull}, for every $K\in\tau_{h_i}$ it holds that
\begin{align*}
\norm{\mathbf{\Pi}_{h_i}(\Psi_i(\bE))}{\hscurl{K}{s}}\!=\!\norm{\br_K(\Psi_i(\bE))}{\hscurl{K}{s}}
\!\leq\! C\norm{\Psi_i(\bE)}{\hscurl{K}{s}}
\!\leq\! C\norm{\bE}{\hscurl{\bT_i(K)}{s}},
\end{align*}
so that
\begin{align*}
\sum\limits_{K\in\tau_{h_i}}\norm{\mathbf{\Pi}_{h_i}(\Psi_i(\bE))}{\hscurl{K}{s}}^2\leq
C^2\norm{\bE}{\hscurl{\D}{s}}^2,
\end{align*}
where $C>0$ follows from \cref{prop:locinterp_cont,lem:smooth_pull}. The estimate follows by an application of \cref{lem:curl_pull_cont}.
\end{proof}
\subsubsection{Fully discrete convergence to an extended solution over \texorpdfstring{$\D_H$}{Lg}}
We continue by presenting the corresponding fully discrete version of \cref{thm:first_result_disc_num},
stating the convergence of the solution of \cref{prob:varprobnumcurv} to a smooth extension of the solution of
\cref{prob:varprob}. For simplicity (as before) we consider only the case of isoparametric finite elements ($\K=k)$.

\begin{theorem}\label{thm:first_result_disc_num}
Let all the assumptions in \cref{thm:first_result_disc} hold and take $\K=k$. Moreover, assume that there holds that
\begin{align*}
\bJ\in\bW^{s,q}(\D^H)\quad\mbox{and}\quad \epsilon_{i,j},\;(\mu^{-1})_{i,j}\in W^{s,\infty}(\D_H)\quad\forall\;i,\;j\in\{1:3\},
\end{align*}
for some $q>\min(2,s/3)$, where $s\in\{1:k\}$ is as in \cref{thm:first_result_disc_num},
as well as the following conditions on the quadrature rules defining $\Phi_h(\cdot, \cdot)$ and $\bF_h(\cdot)$ in \cref{eq:numericPhi,eq:numericrhs}, respectively,
\begin{itemize}
\item $Q^1_{\rK}$ is exact for polynomials
of degree $2k+s-3$ and
\item $Q^2_{\rK}$ and $Q^3_{\rK}$ are exact for polynomials
of degree $3k+s-3$.
\end{itemize}
Then, there exists some $\mathfrak{i}\in\IN$ such that for all
$i\in\IN$ with $i>\mathfrak{i}$, \cref{prob:varprobnumcurv}
is uniquely solvable and its solution, denoted $\widetilde\bE_{h_i}$, satisfies
\begin{align*}
\norm{\bE-\widetilde\bE_{h_i}}{\hcurl{\D_{h_i}}}\!\leq\!
C(\vartheta)(h^{\K}\!+\!h^{r(\K+1)})\norm{\bE}{\hscurl{\D_H}{r}}\!+\!
C(\vartheta,\kT)h_i^s\norm{\bE}{\hscurl{\D}{s}}\!+\!Ch_i^s,
\end{align*}
where positive constants $C(\vartheta)$ and $C(\kT)$ depend only on
$\vartheta$ and $\kT$, respectively, and $C$ is a positive constant
independent of the mesh-size.
\end{theorem}
\begin{proof}
Fix $i\in\IN$ as in the proof of \cref{thm:first_result_disc}. By the triangle inequality, we have that
\begin{align*}
\norm{\bE-\widetilde\bE_{h_i}}{\hcurl{\D_{h_i}}}\leq \norm{\bE-\Psi_i\bE}{\hcurl{\D_{h_i}}}+\norm{\Psi_i\bE-\widetilde\bE_{h_i}}{\hcurl{\D_{h_i}}},
\end{align*}
and the result follows by an application of \cref{lem:trans_approx,thm:pull_result_disc_num} (also see the proof of \cref{thm:first_result_cont}). 
%
\end{proof}
\section{Numerical Results}\label{sec:NumRes}
We test our main results on a simple numerical example. For simplicity, and in order to study only the effects of
domain approximation quality on the convergence rate of the finite element method, we consider the
exact domain $\D$ to be the ball with radius $1$ centered at the origin, i.e.~
$$\D:=\{\bx\in\IR^3\ :\ \norm{\bx}{\IR^3}<1\}.$$
Since $\D$ is a convex domain, the approximate domains in $\kD$ may be chosen to be contained in $\D$, so that
$\D\equiv\D_H$ in \cref{ass:first_approach} and no extension of the solution $\bE\in\hocurl{\D}$
of \cref{prob:varprob} outside of $\D$ is required. Hence, our following results consider the
error measurement as in \cref{ssec:extension_convergence_disc} only.

We consider the coercive variational problem on $\hocurl{\D}$ given by the following sesquilinear and antilinear
forms:
\begin{align*}
\Phi(\bU,\bV):=\int_\D(\mu_0\bI)^{-1}\curl\bU\cdot{\curl\overline{\bV}}-\omega^2(\epsilon_0\bI)\bU\cdot\overline{\bV}\;\d\!\bx\quad\mbox{and}\quad\bF(\bV):=-\imath\omega\int_\D\bJ\cdot\overline{\bV}\;\d\!\bx,
\end{align*}
where $\epsilon_0:=1$, $\mu_0:=2$, $\omega:=1$ and $\bJ:=\imath[J_1,J_2,J_3]^\top$, with
\begin{align*}
J_1(\bx)&:=x_1-\tfrac{\pi^2}{8}x_1x_2\cos\left(\tfrac{\pi}{2}\norm{\bx}{\IR^3}^2\right),\\
J_2(\bx)&:=x_2+\left(\tfrac{1}{4}+\tfrac{\pi^2}{8}(x_1^2+x_3^2)\right)\cos\left(\tfrac{\pi}{2}\norm{\bx}{\IR^3}^2\right)+\tfrac{\pi}{4}\sin\left(\tfrac{\pi}{2}\norm{\bx}{\IR^3}^2\right),\\
J_3(\bx)&:=x_3-\tfrac{\pi^2}{8}x_2x_3\cos\left(\tfrac{\pi}{2}\norm{\bx}{\IR^3}^2\right).
\end{align*}
Under the above choices, the exact solution to \cref{prob:varprob} is 
\begin{align}\label{eq:real_sol}
\bE(\bx):=\begin{bmatrix}
x_1 , x_2 + \frac{1}{4}\cos\left(\frac{\pi}{2}\norm{\bx}{\IR^3}^2\right) , x_3
\end{bmatrix}^\top.
\end{align} 
The various meshes used throughout our experiments were constructed using GMSH \cite{geuzaine2009gmsh},
while \cref{prob:varprob_approx_disc} was solved using GETDP version 3.4.0 \cite{getdp}.
\subsection{Approximate domains}
Let us consider two different sequences of meshes of different order. The first sequence considers meshes constructed
from straight tetrahedrons only ($\K=1$), while the second sequence considers meshes consisting of second order
elements curved tetrahedrons ($\K=2$). \cref{fig:mesh} shows the first three meshes
of first and second order. For more details on the conditions satisfied by the second order mesh elements, we refer to
\cite{Johnen_2013}.

\begin{figure}[]
\centering
\fbox{\includegraphics[width=.2\textwidth]{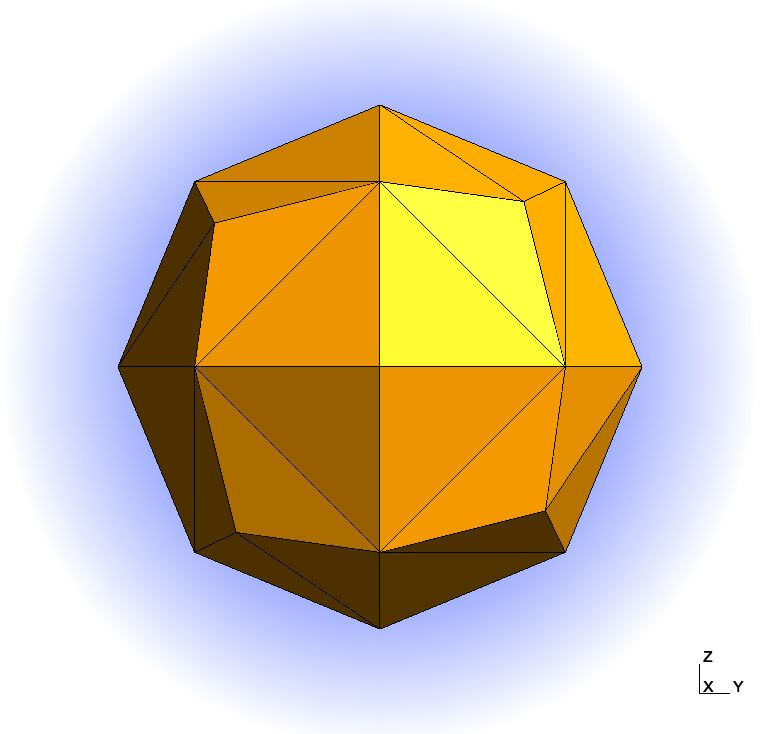}}\hfill
\fbox{\includegraphics[width=.2\textwidth]{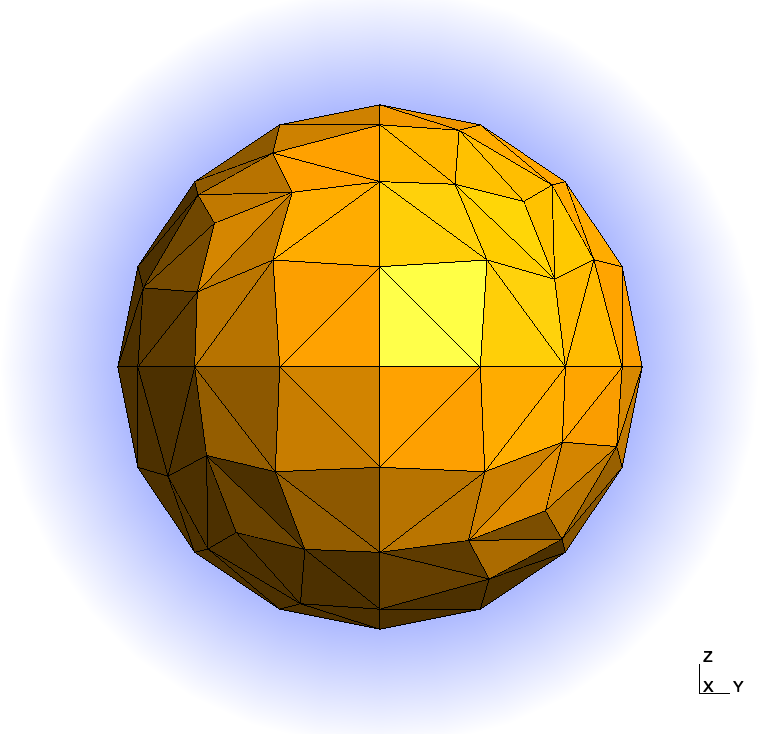}}\hfill
\fbox{\includegraphics[width=.2\textwidth]{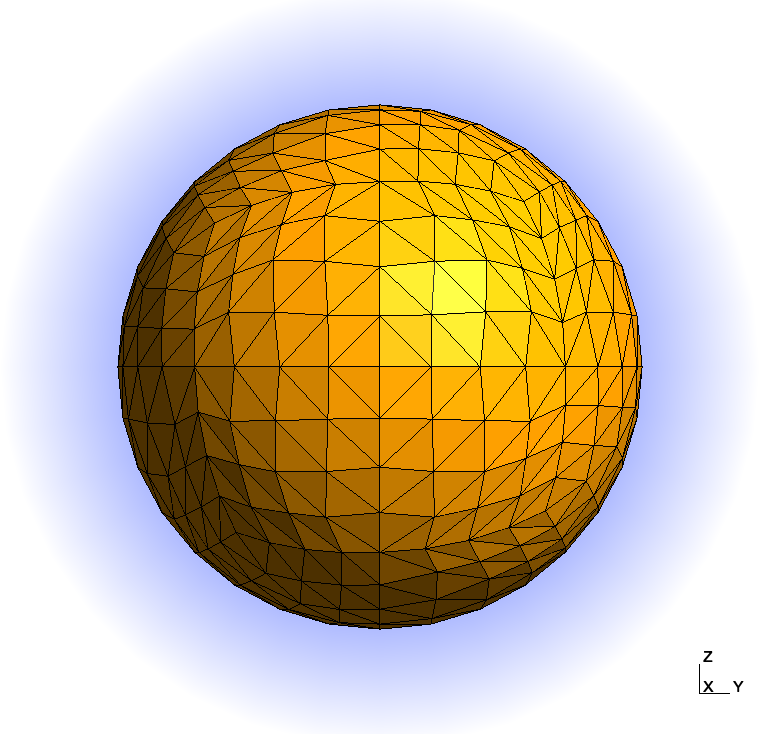}}\\~\\
\fbox{\includegraphics[width=.2\textwidth]{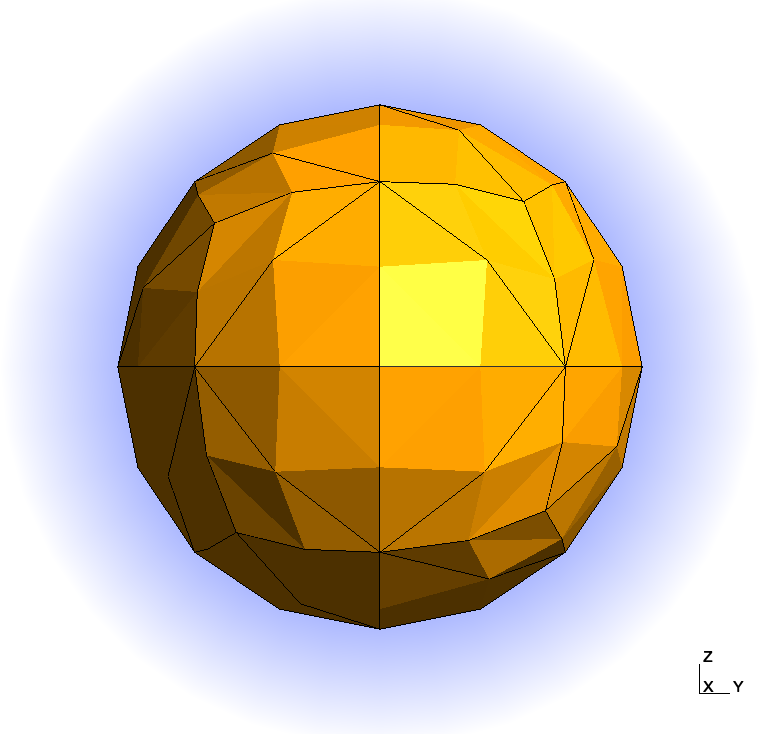}}\hfill
\fbox{\includegraphics[width=.2\textwidth]{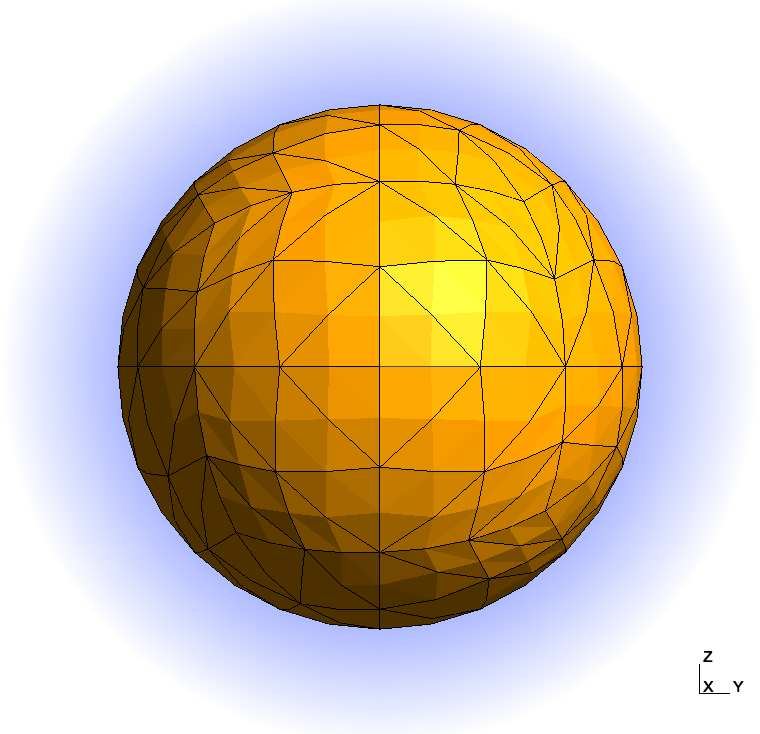}}\hfill
\fbox{\includegraphics[width=.2\textwidth]{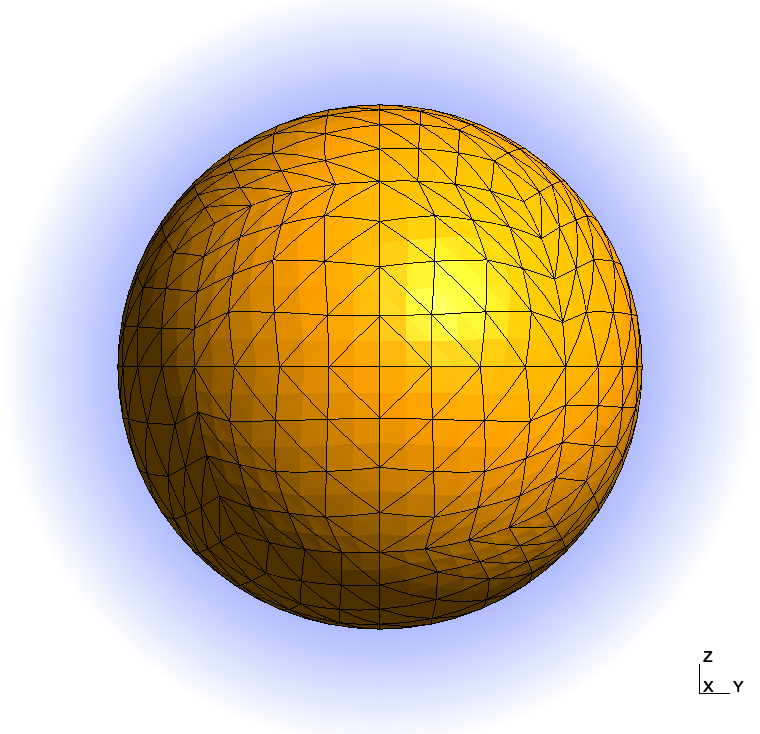}}
\caption{First three meshes of order $\K=1$, followed by the first three meshes of order $\K=2$.}
\label{fig:mesh}
\end{figure}


\subsection{Convergence results}
We employ first and second order curl-conforming elements on both straight and curved (order $2$) meshes
in order to test the results exposed in Theorem \ref{thm:first_result_disc}. To that end, we measure the error
\begin{align*}
\norm{\bE-\bE_{h_i}}{\hcurl{\D_{h_i}}}
\end{align*}
as $i\in\IN$ grows towards infinity, where $\bE\in\hocurl{\D}$ is as in \eqref{eq:real_sol} (the solution to \cref{prob:varprob})
and $\bE_{h_i}\in\hocurl{\D_{h_i}}$ denotes the solution to \cref{prob:varprob_approx_disc}. 
\cref{fig:first_order_poly} displays the convergence of the solution to \cref{prob:varprob_approx_disc}
to the continuous one when
using a first-order curl-conforming approximation ($k=1$) together with first and second order mesh elements ($\K=1$ and $\K=2$, respectively). 
\cref{fig:second_order_poly}, on the other hand, displays the convergence of the solution to \cref{prob:varprob_approx_disc}
to the continuous solution when
using a second-order curl-conforming approximation (i.e.~$k=1$) on the same meshes as before.

\begin{figure}
\centering
{\includegraphics[width=.65\textwidth]{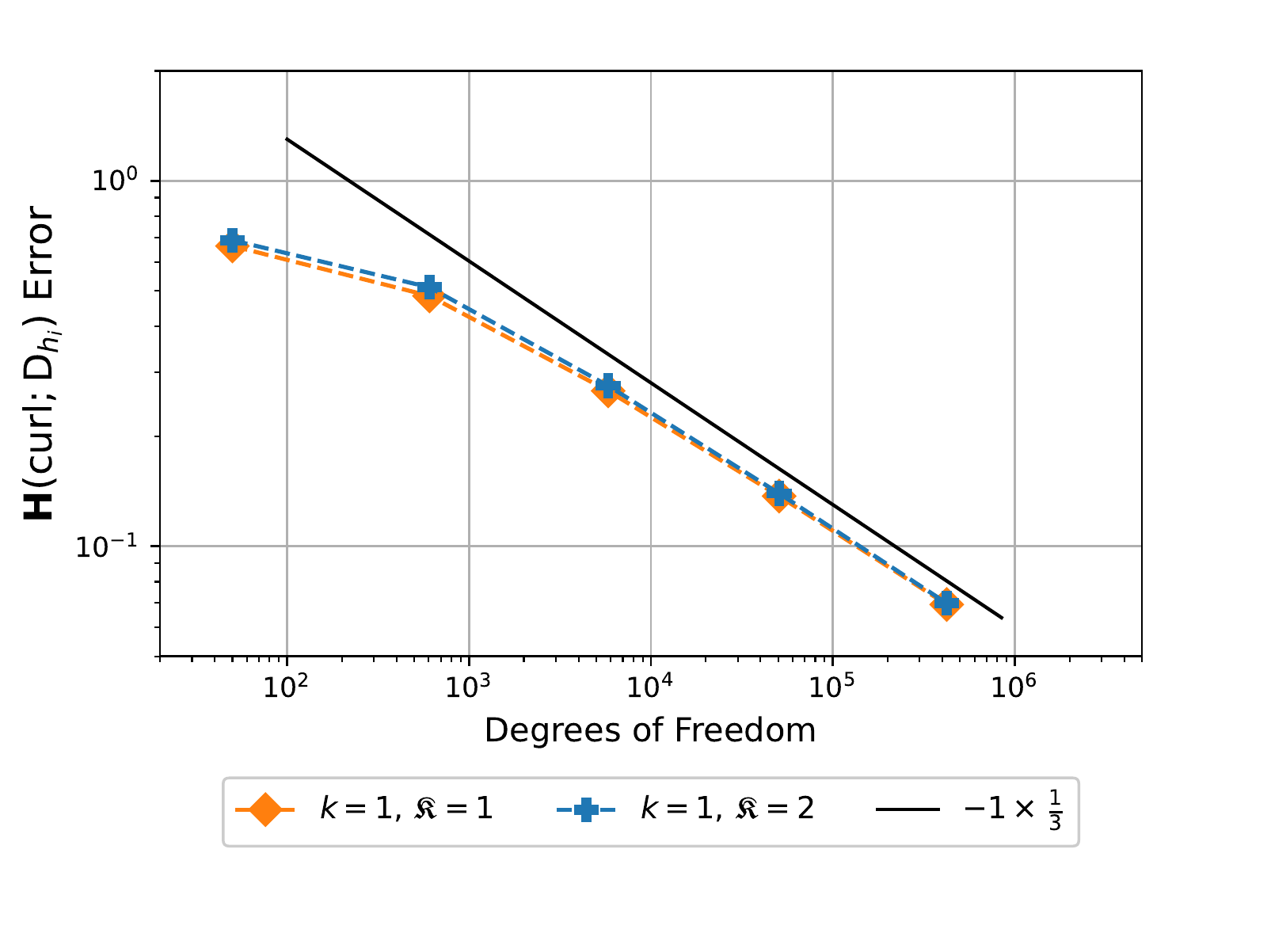}}\hfill
\caption{Error convergence in $\hcurl{\D_{h_i}}$-norm of solutions to \cref{prob:varprob_approx_disc} with respect to that of \cref{prob:varprob} using first order curl-conforming finite elements on straight ($\K=1$) and
curved ($\K=2$) meshes. In both cases, we observe the expected linear behavior with respect to the mesh-size .}
\label{fig:first_order_poly}
\end{figure}
\begin{figure}
\centering
{\includegraphics[width=.65\textwidth]{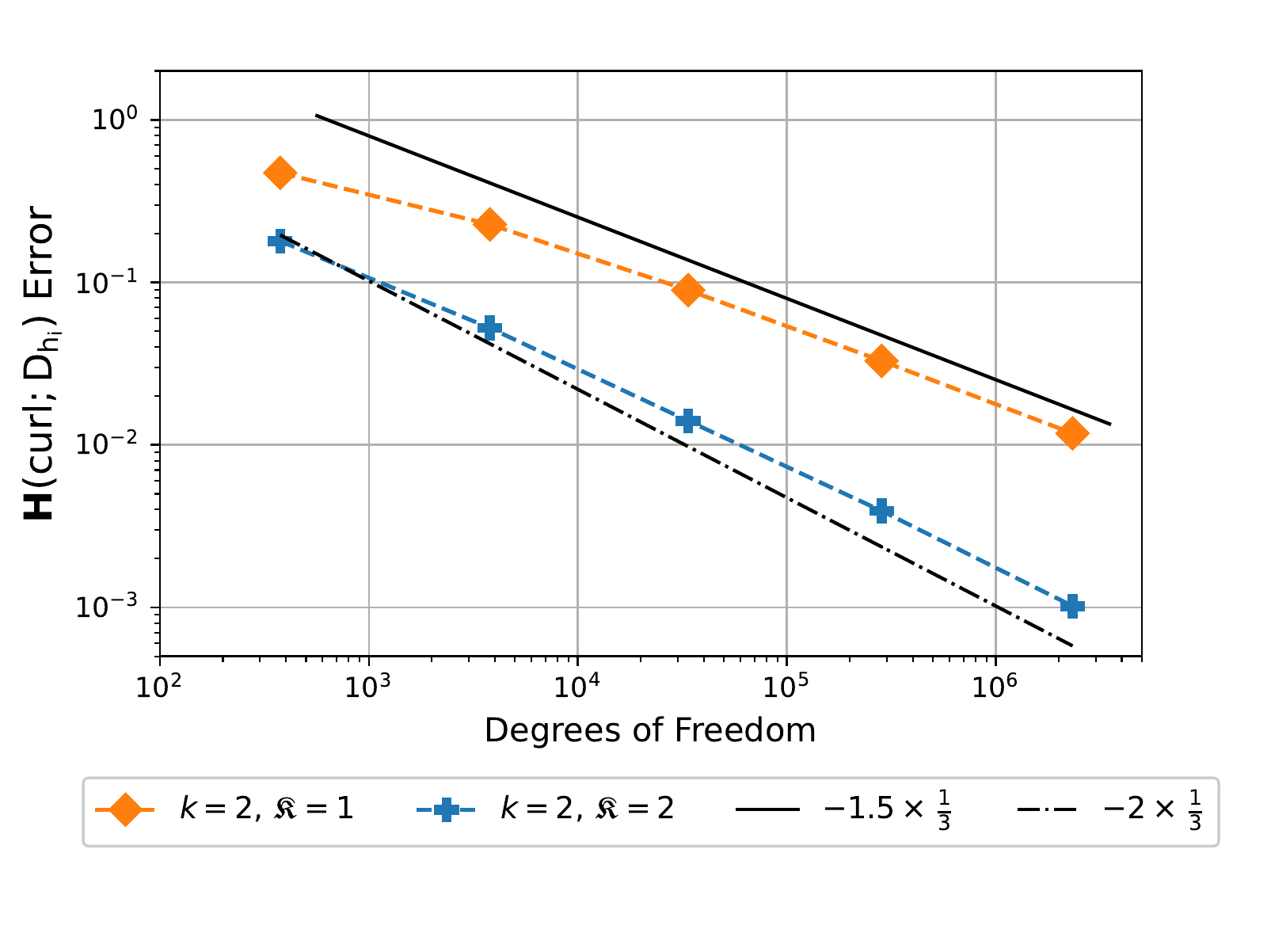}}\hfill
\caption{Error convergence in $\hcurl{\D_{h_i}}$-norm of solutions to \cref{prob:varprob_approx_disc} with respect to that of \cref{prob:varprob} using second-order curl-conforming finite elements on straight ($\K=1$) and
curved ($\K=2$) meshes. A pre-asymptotic regime is observed in both cases after which the latter achieves the expected second order---with respect to $h_i$---convergence rate, while the former case only attains a degenerated rate of roughly order $1.5$ with respect to the mesh-size due to the low-order approximation of the boundary of $\D$.}
\label{fig:second_order_poly}
\end{figure}

\section{Conclusions}\label{sec:Conc}
For the family of Maxwell variational problems here considered, \cref{thm:pull_result_cont,thm:first_result_cont}
provide sufficient conditions on the family of approximate domains $\{\widetilde\D_i\}_{i\in\IN}$
to ensure convergence rates of: (i) pull-backs of continuous solutions in approximate domains to those in the original one; and, (ii)
continuous solutions in approximate domains to smooth extensions of the exact solution.
\cref{thm:pull_result_disc,thm:pull_result_disc} extend these results to their
discrete counterparts to then include the effects of numerical integration for a fully discrete analysis in \cref{thm:pull_result_disc_num,thm:pull_result_disc_num}, based on our previous work \cite{aylwin2020effect}.

Our results on curved meshes established various properties of (local) interpolation on curved meshes
and pull-backs $\psi^c$ in \eqref{eq:fe_spaces_K}. These correspond to \cref{lem:pullbackbound,lem:ref:interp:error} and \cref{prop:locinterp,prop:locinterp_cont}, which are of
independent interest. Also the simple numerical examples in \cref{sec:NumRes} confirm our findings. Observe the failure of second-order polynomials to achieve second-order convergence rates to the solution on straight meshes in \cref{fig:second_order_poly}.

We left out the issue of demonstrating the regularities of the electric field that are required in
order to ensure different rates of convergence (\emph{cf.}~\cite{alberti2014elliptic} for globally smooth boundaries). However, results ensuring
arbitrary degrees of regularity in domains with corners and allowing for an application of the results
as in \cite{Babu_kaRID_ID_Partially_supp_2000}, are, to the best of our knowledge,
unavailable for Maxwell's equations.

Finally, we were able to consider integer degrees of regularity only in
\cref{thm:first_result_disc,thm:first_result_disc_num} ($s\in\{1:k\}$).
This stems from the same deficiency in \cref{lem:pullbackbound} and further
improvements are left as future work as well as extensions to more specific and
varied Maxwell variational problems such as problems in periodic media,
FEM/BEM couplings and applications in uncertainty quantification \cite{aylwin2020properties,AJZS18,silva2017quantifying}.

\section*{Acknowledgments}
The authors would like to thank Dr.~Jos\'e Pinto and Dr.~Fernando Henr\'iquez for their invaluable comments on earlier versions of this manuscript.

\appendix
\section{Technical results concerning curl-conforming finite elements}\label{sec:tech_res}

\begin{lemma}\label{lem:curvineqK}
Let \cref{ass:mesh} hold and take $s\in\{0:\K+1\}$ and $K\in\tau_h$. Then, for any $\breve{U}\in \hsob{s}{\rK}$
and with 
$U:=\breve{U}\circ \bmT_{{K}}^{-1}$, it holds that $U\in\hsob{s}{K}$ and that,
\begin{align}
\vert{\breve U}\vert_{{s},{\rK}}\leq
C\inf_{\bx\in\breve{K}}\modulo{\det{\rd\bmT_K(\bx)}}^{-\frac{1}{2}}h^{s}
\norm{U}{s,K},\label{eq:tech1}
\end{align}
where the constant $C>0$ is independent of $K$ and $h$. Analogously, for any $U\in\hsob{s}{K}$ and with
$\breve {U}:={U}\circ \bmT_{{K}}$,
it holds that $\breve U\in\hsob{s}{\rK}$ and 
\begin{align}
\seminorm{U}{s,K}\leq
C\sup_{\bx\in\breve{K}}\modulo{\det{\rd\bmT_K(\bx)}}^{\frac{1}{2}}h^{-s}
\norm{\breve U}{s,\rK},\label{eq:tech2}
\end{align}
with a positive constant $C$ as before.
\end{lemma}
\begin{proof}
The statement in \cref{eq:tech1} is nothing more than Lemma 1 in \cite{ciarlet1972combined},
while \cref{eq:tech2} follows by replacing $\bmT_K^{-1}$ with $\bmT_K$ in Lemma 1 in \cite{ciarlet1972combined}, together---in both cases---with \Cref{ass:mesh} and \cref{eq:cond:dTK}. Note that Lemma 3 in \cite{ciarlet1972interpolation} ensures that the constants in
\cref{eq:tech1,eq:tech2} depend only on $s$.
\end{proof}

\begin{lemma}\label{lem:pullbackbound}
Let \cref{ass:mesh} hold. For all $K\in\tau_h$ and all
$\bV\in\hscurl{K}{s}$ with $s\in \{0:\K+1\}$, it holds that
\begin{align}\label{eq:pullback:1:estimate}
\modulo{\psi^c_K(\bV)}_{s,\rK}\leq C  h^{s-\half} \norm{\bV}{s,K}\quad\mbox{and} \quad
\modulo{\curl\psi^c_K(\bV)}_{s,\rK}\leq C h^{s+\half}\norm{\curl\bV}{s,K},
\end{align}
for a positive constant $C$ independent of $K\in\tau_h$ and $h$. Also, for all
$\bV\in\hscurl{\rK}{s}$ it holds that
\begin{align*}
\modulo{({\psi^c_K})^{-1}(\bV)}_{s,K}\leq C  h^{\half-s} \norm{\bV}{s,\rK}\quad\mbox{and} \quad
\modulo{\curl({\psi^c_K})^{-1}(\bV)}_{s,K}\leq C h^{-(s+\half)}\norm{\curl\bV}{s,\rK},
\end{align*}
for a positive constant $C$ as before.
\end{lemma}
\begin{proof}
We will prove only the estimates in \cref{eq:pullback:1:estimate}, as those for the inverse
of the pull-back follow analogously by noticing that
\begin{align*}
({\psi^c_K})^{-1}(\bV)=\rd\bmT_K^{-\top}(\bV\circ\bmT_K^{-1})\quad\mbox{and}\quad
\curl({\psi^c_K})^{-1}(\bV)=(\rd\bmT_K^{-1})^\co\curl\bV\circ\bmT_K^{-1},
\end{align*}
showcasing the same structure and satisfying analogous properties, despite the different
signs in the powers of $h$.

Take $\bV\in\Hsob{s}{K}$ for any $K\in\tau_h$ and $s\in\IN_0$.
Let $\bA_K=(a_{ij})_{i,j=1}^3\in\bm{W}^{\infty,\K+1}(K;\IC^{3\times 3})$, be either $\rd \bmT_K^\top$ or $\rd \bmT_K^\co$.
By definition, it holds that
\begin{align*}
\modulo{\bA_K\bV\circ \bmT_K}_{s,\rK}=\left(\sum\limits_{i=1}^3\modulo{\sum\limits_{j=1}^3a_{ij}\bV_j\circ \bmT_K}_{s,\rK}^2\right)^\half.
\end{align*}
Moreover, for $i\in\{1:3\}$, by Titu's lemma \cite[Sec.~1.2]{Andreescu_2012} we have that
\begin{align*}
\modulo{\sum\limits_{j=1}^3a_{ij}\bV_j\circ \bmT_K}_{s,\rK}^2\leq 3\sum\limits_{j=1}^3\modulo{a_{ij}\bV_j\circ \bmT_K}_{s,\rK}^2,
\end{align*}
and, for any pair $i,j\in\{1:3\}$, it holds that
\begin{align}
\modulo{a_{ij}\bV_j\circ \bmT_K}_{s,\rK}^2\!=\!\sum_{\modulo{\balpha}=s}\left\Vert{\frac{\partial^\balpha}{\partial\bx^\balpha} (a_{ij}\bV_j\circ \bmT_K)}\right\Vert_{0,\rK}^2
\leq C\sum\limits_{\substack{\modulo{\balpha_1}\leq s\\\modulo{\balpha_2}=s-\modulo{\balpha_1}}}\left\Vert{\frac{\partial^{\balpha_1}}{\partial\bx^{\balpha_1}} a_{ij}\frac{\partial^{\balpha_2}}{\partial\bx^{\balpha_2}}\bV_j\circ \bmT_K}\right\Vert_{0,\rK}^2&\label{eq:Lem2:ineqC}\\
\leq C\sum\limits_{\substack{\modulo{\balpha_1}\leq s\\\modulo{\balpha_2}=s-\modulo{\balpha_1}}}\left\Vert{\frac{\partial^{\balpha_1}}{\partial\bx^{\balpha_1}} a_{ij}}\right\Vert_{\lp{\infty}{\breve K}}^2\left\Vert{\frac{\partial^{\balpha_2}}{\partial\bx^{\balpha_2}}\bV_j\circ \bmT_K}\right\Vert_{0,\rK}^2&,\nonumber
\end{align}
where the positive constant $C$ in \cref{eq:Lem2:ineqC} follows from the
Cauchy-Schwarz inequality and depends only on $s$.
For the particular choice $\bA_k=\rd \bmT_K^{ \top}$, \cref{eq:normderbound}
ensures the existence of a uniform constant $C>0$, independent
of $K$ and the mesh-size, such that
\begin{align*}
\sum\limits_{\substack{\modulo{\balpha_1}\leq s\\\modulo{\balpha_2}=s-\modulo{\balpha_1}}}\left\Vert{\frac{\partial^{\balpha_1}}{\partial\bx^{\balpha_1}} a_{ij}}\right\Vert_{\lp{\infty}{\breve K}}^2\left\Vert{\frac{\partial^{\balpha_2}}{\partial\bx^{\balpha_2}}\bV_j\circ \bmT_K}\right\Vert_{0,\rK}^2
&\leq C \sum\limits_{\substack{\modulo{\balpha_1}\leq s\\\modulo{\balpha_2}=s-\modulo{\balpha_1}}}\hspace{-0.2cm}h^{2(\modulo{\balpha_1}+1)}\left\Vert{\frac{\partial^{\balpha_2}}{\partial\bx^{\balpha_2}}\bV_j\circ \bmT_K}\right\Vert_{0,\rK}^2\\
=C \sum\limits_{m=0}^{s}h^{2m+2}\modulo{\bV_j\circ \bmT_K}_{s-m,\rK}^2
&\leq C \left(\inf\limits_{\bx\in\breve K}{\det{\rd \bmT_K(\bx)}}\right)^{-1}h^{2s+2}\norm{\bV_j}{s,K}^2,
\end{align*}
where $C$ now includes the constant in
\cref{lem:curvineqK}. The
estimate for $\psi^c_K(\bV)$ is retrieved by the bound
in \cref{eq:bounddet}.

For $\bA_K=\rd \bmT_K^\co$, we proceed analogously
\begin{align}
\sum\limits_{\substack{\modulo{\balpha_1}\leq s\\\modulo{\balpha_2}=s-\modulo{\balpha_1}}}\!\left\Vert{\frac{\partial^{\balpha_1}}{\partial\bx^{\balpha_1}} a_{ij}}\right\Vert_{\lp{\infty}{\breve K}}^2\!\left\Vert{\frac{\partial^{\balpha_2}}{\partial\bx^{\balpha_2}}\bV_j\circ \bmT_K}\right\Vert_{0,\rK}^2\!
&\leq\! C\!\!\! \sum\limits_{\substack{\modulo{\balpha_1}\leq s\\\modulo{\balpha_2}=s-\modulo{\balpha_1}}}\hspace{-0.2cm}h^{2(\modulo{\balpha_1}+2)}\!\left\Vert{\frac{\partial^{\balpha_2}}{\partial\bx^{\balpha_2}}\bV_j\circ \bmT_K}\right\Vert_{0,\rK}^2\label{eq:lemma2:cof:bound}\\
=C \sum\limits_{m=0}^{s}h^{2m+4}\modulo{\bV_j\circ \bmT_K}_{s-m,\rK}^2
&\leq C \left(\inf\limits_{\bx\in\breve K}{\det{\rd \bmT_K}}\right)^{-1}h^{2s+4}\norm{\bV_j}{s,K}^2,\nonumber
\end{align}
where \cref{eq:lemma2:cof:bound} is deduced from \cref{eq:normderbound} and the cofactor matrix definition. The bound for $\curl\psi^c_K(\bV)$ then follows as before.
\end{proof}

\begin{lemma}\label{lem:ref:interp:error}
Let \cref{ass:mesh} hold and let $\bU\in\hscurl{\rK}{s}$ for $s\in\{1:k\}$. Then, for any $l\in\{1:s\}$ it holds that,
\begin{align*}
\norm{\bU-\breve\br(\bU)}{l,\rK}&\leq C \left(\seminorm{\bU}{s,\rK}+\seminorm{\curl\bU}{s,\rK}\right)\quad\mbox{and}\quad
\norm{\curl\bU-\curl\breve\br(\bU)}{l,\rK}\leq C \seminorm{\curl\bU}{s,\rK},
\end{align*}
for a positive constant $C>0$ independent of $\bU\in\hscurl{\rK}{s}$.
\end{lemma}

\begin{proof}
Take $s\in\{1:k\}$ and $l\in\{1:s\}$. We will prove only the estimate $$\norm{\bU-\breve\br(\bU)}{l,\rK}\leq C \left(\seminorm{\bU}{s,\rK}+\seminorm{\curl\bU}{s,\rK}\right),$$ since the remaining estimate follows by analogous arguments.
Let $\bU\in\hscurl{\rK}{s}$ and take any $\bm{\phi}\in\mathbb{P}_{k-1}(\rK;\IC^3)$. Then, by the invariance of the
canonical interpolation operator, we have that
\begin{align}\label{eq:smooth_error_inter_1}
\norm{\bU-\breve{\br}(\bU)}{l,\rK}=\norm{(\Id-\breve{\br})(\bU+\bm\phi)}{l,\rK}\leq \norm{\bU+\bm\phi}{l,\rK}+\norm{\breve{\br}(\bU+\bm\phi)}{l,\rK}.
\end{align}
Lemma 5.38 in \cite{Monk:2003aa} then allows us to bound the degrees of freedom of $(U+\bm\phi)$ through the $\hscurl{\rK}{s}$-norm (since $s\geq 1$). Specifically, we have that
\begin{align}\label{eq:smooth_error_inter_2}
\norm{\breve{\br}(\bU+\bm\phi)}{l,\rK}\leq C\left(\norm{\bU+\bm\phi}{s,\rK}+\norm{\curl(\bU+\bm\phi)}{s,\rK}\right),
\end{align}
where $C$ may depend on $l\in\{1:s\}$ and $s\in\{1:k\}$, but is independent of $\bU$ and $\bm\phi$.
A combination of \cref{eq:smooth_error_inter_1,eq:smooth_error_inter_2}, together with the fact that $l\leq s$, yields
\begin{align} \label{eq:smooth_error_inter_3}
\norm{\bU-\breve{\br}(\bU)}{l,\rK}\leq C\left(\norm{\bU+\bm\phi}{s,\rK}+\norm{\curl(\bU+\bm\phi)}{s,\rK}\right),
\end{align}
where $C$ is not necessarily the same as before, but is still independent of $\bU$ and $\bm\phi$.
Since \cref{eq:smooth_error_inter_3} holds for any $\bm{\phi}\in\mathbb{P}_{k-1}(\rK;\IC^3)$ we may take the infimum of its
right-hand side over $\mathbb{P}_{k-1}(\rK;\IC^3)$, which we may then bound by \cite[Thm.~5.5]{Monk:2003aa}, to obtain
\begin{align*} 
\norm{\bU-\breve{\br}(\bU)}{l,\rK}\leq C\left(\seminorm{\bU}{s,\rK}+\seminorm{\curl\bU}{s,\rK}\right),
\end{align*}
for a positive constant $C$ as before. 

The estimate
\begin{align*}
\norm{\curl\bU-\curl\breve\br(\bU)}{l,\rK}&\leq C \seminorm{\curl\bU}{s,\rK},
\end{align*}
follows by analogous arguments employing Lemma 5.40 in \cite{Monk:2003aa} and \cite[Lem.~5.15]{Monk:2003aa} in lieu of \cite[Lem.~5.38]{Monk:2003aa}.
\end{proof}

\section{Proofs of \texorpdfstring{\cref{prop:locinterp,prop:locinterp_cont}}{Lg}}
\label{sec:tech_res2}
\begin{proof}[Proof of \cref{prop:locinterp}]
For $\bU\in\hscurl{K}{s}$, we first estimate the
$\bm{L}^2$-portion of the norm:
\begin{align}\label{eq:interp:L2:1}
\norm{\bU-\br_K(\bU)}{0,K}\leq\sup\limits_{\bx\in\breve K}\det{\rd \bmT_K(\bx)}^{\half}\sup\limits_{\bx\in\breve K}\norm{\rd\bmT^{-1}(\bx)}{}\norm{\psi^c_K(\bU)-\breve{\br}(\psi^c_K(\bU))}{0,\rK}.
\end{align}
By \Cref{lem:ref:interp:error}, one has
\begin{align*}
\norm{\psi^c_K(\bU)-\breve{\br}(\psi^c_K(\bU))}{0,\rK}\leq c\left(\modulo{\psi^c_K(\bU)}_{s,\rK}+\modulo{\curl\psi^c_K(\bU)}_{s,\rK}\right),
\end{align*}
where $c$ is a positive constant independent of $K\in\tau_h$ and $h$. By \Cref{lem:pullbackbound}, it holds that
\begin{align*}
\norm{\psi^c_K(\bU)-\breve{\br}(\psi^c_K(\bU))}{0,{\rK}}\leq c\left(h^{s-\half}\norm{\bU}{s,K}+h^{s+\half}\norm{\curl\bU}{s,K}\right),
\end{align*}
and combining the last equation with \cref{eq:cond:dTK,eq:bounddet,eq:interp:L2:1} yields the estimate
\begin{align*}
\norm{\bU-\br_K(\bU)}{0,K}\leq c \left(h^{s}\norm{\bU}{s,K}+h^{s+1}\norm{\curl\bU}{s,K}\right)\leq ch^s\norm{\bU}{\hscurl{K}{s}},
\end{align*}
where the positive constant $c$ is as before. We continue with the estimate for the curl. Proceeding as before, we have that
\begin{align*}
\norm{\curl\bU-\curl\br_K(\bU)}{0,K}
\!&\leq\!\sup\limits_{\bx\in\breve K}\det{\rd \bmT_K(\bx)}^{-\half}\sup\limits_{\bx\in\breve K}\norm{\rd\bmT(\bx)}{}\norm{\curl\psi^c_K(\bU)-\curl\breve{\br}(\psi^c_K(\bU))}{0,K}\\
&\leq c\sup\limits_{\bx\in\breve K}\det{\rd \bmT_K(\bx)}^{-\half}\sup\limits_{\bx\in\breve K}\norm{\rd\bmT(\bx)}{}\seminorm{\curl\psi^c_K(\bU)}{s,\rK},
\end{align*}
where the last inequality follows from \Cref{lem:ref:interp:error}. \Cref{lem:pullbackbound}, together with \cref{eq:cond:dTK} and \cref{eq:bounddet}, leads to
\begin{align*}
\norm{\curl\bU-\curl\br_K(\bU)}{0,K}\leq ch^{s}\norm{\rd\bmT(\bx)}{}\norm{\curl\bU}{s,K}.
\end{align*}
The combination of the $\bm{L}^2$- and curl-estimates yield the
approximation result.
\end{proof}

\begin{proof}[Proof of \cref{prop:locinterp_cont}]
Take $\bU\in\hscurl{K}{s}$. Then, it holds that
\begin{align*}
\norm{\br_K(\bU)}{\hscurl{K}{s}}\leq \norm{\br_K(\bU)-\bU}{\hscurl{K}{s}}+\norm{\bU}{\hscurl{K}{s}}.
\end{align*}
Moreover,
\begin{align}\label{eq:lem:rk:bound:1}
\norm{\br_K(\bU)-\bU}{\hscurl{K}{s}}=\norm{({\psi^c_K})^{-1}(\breve\br(\psi^c_K(\bU))-\psi^c_K(\bU))}{\hscurl{K}{s}}.
 \end{align}
From \Cref{lem:pullbackbound}, for $l\in\{0:s\}$, it follows that
\begin{align*}
\seminorm{({\psi^c_K})^{-1}(\breve\br(\psi^c_K(\bU))-\psi^c_K(\bU))}{l,K}\leq c  h^{\half-l} \norm{\breve\br(\psi^c_K(\bU))-\psi^c_K(\bU)}{l,\rK},
\end{align*}
and a sequential application of \Cref{lem:ref:interp:error,lem:pullbackbound} yields
\begin{align*}
\norm{\breve\br(\psi^c_K(\bU))-\psi^c_K(\bU)}{l,\rK}
\leq c \left(h^{s-\half} \norm{\bU}{s,K}+h^{s + \half} \norm{\curl\bU}{s,K}\right),
\end{align*}
so that
\begin{align}\label{eq:lem:rk:bound:2}
\norm{\br_K(\bU)-\bU}{s,K}\leq c\norm{\bU}{\hscurl{K}{s}}.
\end{align}
We may proceed analogously and bound the curl portion of the norm as
\begin{align}\label{eq:lem:rk:bound:3}
\norm{\curl\br_K(\bU)-\curl\bU}{\Hsob{s}{K}}\leq c\norm{\curl\bU}{s,K}.
\end{align}
Then, combining the results in \cref{eq:lem:rk:bound:1,eq:lem:rk:bound:2,eq:lem:rk:bound:3} completes the proof.
\end{proof}

\section{Proofs of \texorpdfstring{\cref{lem:error_infty,lem:trans_approx}}{Lg}}\label{sec:tech_res_T}
\begin{proof}[Proof of \cref{lem:error_infty}]
Take $U$ as a smooth function in $\Upsilon$, i.e.~$u\in\C^{\infty}(\Upsilon)$. Then, for all $x\in\Omega$, it holds that
\begin{align*}
U\circ\bT(\bx)-U(\bx)=\int_0^1\nabla U((1-t)\bx+t\bT(\bx))\cdot(\bT(\bx)-\bx)\d\! t.
\end{align*}
Observe that $(1-t)\bx+t\bT(\bx)\in\Upsilon$, for all $t\in [0,1]$, and define $\bT_t(\bx):=(1-t)\bx+t\bT(\bx)$. Then, the convexity of $\Upsilon$ implies $\bT_t(\bx)\in\D_H$ for all $\bx\in\D_H$. Moreover, one has
\begin{align*}
\modulo{U\circ\bT(\bx)-U(\bx)}&=\modulo{\int_0^1\nabla U(\bT_t(\bx))\cdot ({\bT(\bx)-\bx})\d\! t}\\
&\leq\norm{\bT-\Id}{\Lp{\infty}{\Upsilon}}{\int_0^1\norm{\nabla U(\bT_t(\bx))}{\Lp{\infty}{\Omega}}\d\! t}\leq\norm{\bT-\Id}{\Lp{\infty}{\Upsilon}}\norm{U}{W^{1,\infty}(\Upsilon)}.
\end{align*}
The statement follows by taking the supremum over $\bx\in\Omega$ and by density of $\C^\infty(\Upsilon)$ in $W^{1,\infty}(\Upsilon)$.
\end{proof}

\begin{proof}[Proof of \cref{lem:trans_approx}]
We start by proving the statement for $s=0$. Set $\bT$ as required, then
 for any $\bU\in\Lp{2}{\Upsilon}$ it holds that
\begin{align*}
\norm{\bU\circ \bT-\bU}{0,{\Omega}}&\leq\norm{\bU}{0,{\Omega}}+\norm{\bU\circ \bT}{0,{\Omega}}\leq \norm{\bU}{0,{\Omega}}+\vartheta^\half\norm{\bU}{0,{\bT(\Omega)}}\leq(\vartheta^\half+1)\norm{\bU}{0,{\Upsilon}}.
\end{align*}

Now we prove for $s=1$. Take $\bU$ as a smooth function in $\Upsilon$, i.e.~$\bU\in\bm\C^\infty(\Upsilon)$. Then, for all $\bx\in\Omega$ one has
\begin{align*}
\bU\circ\bT(\bx)-\bU(\bx)=\int_0^1\rd\bU((1-t)\bx+t\bT(\bx))\cdot(\bT(\bx)-\bx)\dd t.
\end{align*}
Note that $(1-t)\bx+t\bT(\bx)\in\Upsilon$ for all $t\in [0,1]$ and $\bT_t(\bx):=(1-t)\bx+t\bT(\bx)$
satisfies all the conditions in \cref{eq:t_cond_lem}. In particular,
we have that
\begin{align*}
\norm{\bx-\by}{\IR^3}&\leq \norm{\bT_t(\bx)-\bT_t(\by)}{\IR^3}+t\norm{(\bT(\bx)-\bx)-(\bT(\by)-\by)}{\IR^3}\\
&\leq \norm{\bT_t(\bx)-\bT_t(\by)}{\IR^3}+t\kappa\norm{\bx-\by}{\IR^3},
\end{align*}
implying the invertibility of $\bT_t:\Omega\to\bT_t(\Omega)$.
Moreover,
\begin{align*}
&\norm{\bU\circ\bT-\bU}{0,{\Omega}}^2=\int_{\Omega}\norm{\bU\circ\bT(\bx)-\bU(\bx)}{\IR^3}^2\ddx=\int_{\Omega}\left\Vert{\int_0^1\rd\bU(\bT_t(\bx))\cdot(\bT(\bx)-\bx)\dd t}\right\Vert_{\IR^3}^2\ddx\\
&\leq\int_\Omega\int_0^1\norm{\rd\bU(\bT_t(\bx))\cdot(\bT(\bx)-\bx)}{\IR^3}^2\dd t\ddx\quad\mbox{(Jensen's inequality)}\\
&\leq \norm{\bT-\Id}{\Lp{\infty}{\Upsilon}}^2\int_\Omega\int_0^1\norm{\rd\bU(\bT_t(\bx))}{\IR^{3\times 3}}^2\dd t\ddx\\
&=\norm{\bT-\Id}{\Lp{\infty}{\Upsilon}}^2\int_0^1\int_{\bT_t(\Omega)}\norm{\rd\bU(\bx)}{\IR^{3\times 3}}^2\det{\rd\bT_t(\bx)}^{-1}\ddx \dd t \leq\vartheta\norm{\bT-\Id}{\Lp{\infty}{\Upsilon}}^2\norm{\bU}{1,{\Upsilon}}^2,
\end{align*}
so that 
\begin{align*}
\norm{\bU\circ\bT-\bU}{0,{\Omega}}\leq\vartheta^\half\norm{\bT-\Id}{\Lp{\infty}{\Upsilon}}\norm{\bU}{1,{\Upsilon}}\leq(\vartheta^\half+1)\norm{\bT-\Id}{\Lp{\infty}{\Upsilon}}\norm{\bU}{1,{\Upsilon}}.
\end{align*}
The statement for $s=1$ follows by density of $\bm\C^\infty(\Upsilon)$ in $\Hsob{1}{\Upsilon}$.
The result for real $s\in (0,1)$ follows by applying real interpolation in Sobolev
spaces (\emph{cf.}~\cite[Lem.~22.3]{Tartar2007}).
\end{proof}

\bibliographystyle{plain}
\bibliography{max}
\end{document}